\def\obs#1{{\bf (*** #1 ***)} }
\def\obs#1{}     
\documentclass[twoside,letterpaper,draft,11pt]{amsart}
\usepackage{amsfonts,amssymb,amsmath,amsthm,latexsym,xspace,amscd,amssymb,mathtools,color,enumerate}
\usepackage[mathscr]{eucal}
\usepackage[all]{xy}
\usepackage[english]{babel}
\usepackage[x11names]{xcolor}

\makeatletter\renewcommand\theenumi{\@roman\c@enumi}\makeatother
\theoremstyle{thm}
\newtheorem{thm}{Theorem}[section]
\newtheorem{cor}[thm]{Corollary}
\theoremstyle{lem}

\newtheorem{prop}[thm]{Proposition}

\theoremstyle{rem}
\newtheorem{rem}[thm]{Remark}
\newtheorem{exe}[thm]{Example}

\newcommand{\G}{\mathcal{G}}

\newcommand{\Hh}{\operatorname{T}}
\newcommand{\Sh}{\operatorname{S}}

\newcommand{\I}{\mathcal{I}}

\newcommand{\af}{\alpha}

\newcommand{\bt}{\beta}

\newcommand{\om}{\omega}
\newcommand{\ot}{\otimes}

\newcommand{\m}{^{-1}}
\newcommand{\id}{\operatorname{id}}

\newcommand{\vd}{\vspace{.2cm}}

\newcommand{\anula}{\operatorname{ann}}

\begin{document}
	
\thispagestyle{empty}

\title[The commutative inverse semigroup of the partial abelian extensions]{The commutative inverse semigroup of partial abelian extensions}

\author[D. Bagio]{Dirceu Bagio}
\address{Departamento de Matem\'atica, Universidade Federal de Santa Maria,
	97105-900, Santa Maria, RS, Brazil} \email{bagio@smail.ufsm.br}

\author[A. Ca\~nas]{Andr\'es Ca\~nas}
\address{Escuela de Matematicas, Universidad Industrial de Santander, Cra. 27 Calle 9  UIS
	Edificio 45, Bucaramanga, Colombia}\email{ansecape@correo.uis.edu.co }

\author[V. Mar\'{\i}n]{V\'{\i}ctor Mar\'{\i}n}
\address{Departamento de Matem\'{a}ticas y Estad\'{i}stica, Universidad del Tolima, Santa Helena\\
	Ibagu\'{e}, Colombia} \email{vemarinc@ut.edu.co}

\author[A. Paques]{Antonio Paques }
\address{Instituto de Matem\'atica e Estat\'istica, Universidade Federal do Rio Grande do Sul, 91509-900\\
	Porto Alegre-RS, Brazil}
\email{paques@mat.ufrgs.br}

\author[H. Pinedo]{H\'{e}ctor Pinedo}
\address{Escuela de Matematicas, Universidad Industrial de Santander, Cra. 27 Calle 9  UIS
	Edificio 45, Bucaramanga, Colombia}\email{hpinedot@uis.edu.co}

\date{\today}

\begin{abstract}
	This paper  is a new contribution to  the partial Galois theory of groups. First, given a unital partial action  $\alpha_G$ of a finite group $G$ on an algebra $S$ such that $S$ is an $\af_G$-partial Galois extension of $S^{\af_G}$  and  a normal subgroup $H$ of $G$,  we prove that $\af_G$ induces a unital partial action $\af_{G/H}$ of $G/H$ on the subalgebra of invariants $S^{\af_H}$ of $S$ such that $S^{\alpha_H}$ is an $\alpha_{G/H}$-partial Galois extension of $S^{\af_G}$. Second,  assuming that $G$ is abelian, we construct a commutative inverse semigroup $T_{par}(G,R)$, whose elements are  equivalence classes of  $\af_G$-partial abelian extensions of a commutative algebra  $R$.
	We also prove that there exists a group isomorphism between $T_{par}(G,R)/\rho$ and $T(G,R)$, where $\rho$ is a congruence on  $T_{par}(G,R)$ and  $T(G,R)$ is the classical Harrison group of  the $G$-isomorphism classes of the abelian extensions of $R$. It is shown that the study of $T_{par}(G,R)$  reduces  to the case where $G$ is cyclic. The set of idempotents of  $T_{par}(G,R)$ is also investigated.
\end{abstract}

\maketitle

\setcounter{tocdepth}{1}

\section{Introduction}

In the 1960's, M. Auslander and O. Goldman introduced in \cite{AG} the notion of Galois extension for commutative rings. After that, S. U. Chase, D. K. Harrison and A. Rosenberg developed in \cite{CHR} a Galois theory for commutative rings extending the classical theory over fields. One of the main results  in \cite{CHR} is Theorem 2.3 which has two parts  as described in the sequel. Let $S\supset R$ be a Galois extension of commutative rings with  a finite group $G$ as Galois group (shortly, a $G$-extension of $R$).
The first part establishes a bijective correspondence between  all the subgroups of $G$ and the subrings of $S$ which are $G$-strong and separable as $R$-algebras. In the second part, it was shown that if $H$ is a normal subgroup of $G$ then the subring $S^H$ of $S$ of  the invariants by the action of $H$ is a $G/H$-extension of $R$.

\smallbreak

Given a commutative ring extension $T\supset R$ and a finite group $G$ we say that $T$ is  an {\it abelian} $G$-extension of $R$  if $T$ is a $G$-extension of $R$ and $G$ is abelian. Given  a finite abelian group $G$ and a commutative algebra $R$,  D. K. Harrison constructed in \cite{H}  the group $\Hh(G,R)$, which is called the  {\it Harrison group} of the abelian $G$-extensions of $R$. The elements of $\Hh(G,R)$  are $G$-isomorphism classes of abelian $G$-extensions of $R$. The  binary operation in $\Hh(G,R)$ is defined in the following way. Let $T$ and $T'$ be abelian $G$-extensions of $R$. Thus, $T\otimes_RT'$ is  an abelian ($G\times G$)-extension of $R\otimes_R R\simeq R$, where the action is given by: $(g,h)\big((a\otimes b)\big)=g(a)\otimes h(b)$, for all $g,h\in G$, $a\in T$ and $b\in T'$. Consider the subgroup $\delta G=\{(g,g\m)\,:\,g\in G\}$ of $G\times G$. By the second part of Theorem 2.3 of \cite{CHR}, $(T\otimes_R T')^{\delta G}$ is  an abelian $((G\times G)/\delta G)\simeq G$-extension of $R$ and whence its equivalence class is an element of $\Hh(G,R)$. Therefore, we have an operation on $\Hh(G,R)$. In \cite{H}, it was proved  that $\Hh(G,R)$ is an abelian group.

\smallbreak

The purpose of this paper is to  explore such a construction due to D. K. Harrison in the context of partial Galois extensions. For this, we need first to extend the second part of Theorem 2.3 of \cite{CHR} to the context of partial group actions.

\smallbreak

 In \cite{DFP},  a Galois theory of commutative rings for partial group actions was developed. The results proved in \cite{DFP} generalize many results of \cite{CHR}. Particularly, Theorem 5.1 of \cite{DFP} is the version of the first part of Theorem 2.3 of \cite{CHR} in the context of partial actions. However, a version for the second part of such a theorem was not considered in \cite{DFP}. In our first main result, which is presented below, we provide this.

\smallbreak

\noindent{\bf Theorem \ref{teofundpar2}.} Let $R,S$ be algebras, $G$ a finite group, $\alpha_G$ a  unital partial action of $G$ on $S$ such that $S$ is  a partial $\alpha_G$-extension of $R$ and $H$ a normal subgroup of $G$. Then  $\af_G$ induces a unital partial action $\alpha_{G/H}$ of the quotient group $G/H$ on the subalgebra of invariants $S^{\alpha_H}$ such that $S^{\alpha_H}$ is   a partial $\alpha_{G/H}$-extension of $R$ and $\big(S^{\af_H}\big)^{\af_{G/H}}=R$.

\smallbreak

The above theorem allows to  reproduce the Harrison's construction for partial actions. Given a finite abelian group $G$ and a commutative algebra $R$, we consider the set  $\Hh_{par}(G,R)$ of the $G$-isomorphism classes of (unital) partial abelian  $\af_G$-extensions of $R$. Let  $\alpha_G$ and $\alpha'_G$ be unital partial actions of $G$ on the algebras $S$ and $S'$ respectively. Suppose that $S$ (resp. $S'$) is a partial abelian $\af_G$-extension (resp. $\af'_G$-extension) of $R$. By Proposition 2.9 of \cite{DPP},  $\af_{G\times G}=\af_G\otimes \af'_G$ is a unital partial action of $G\times G$ on $S\otimes_R S'$,  where $\af_{(g,h)}:=\af_g\otimes \af'_h$, for all $g,h\in G$. Moreover, $S\otimes_R S'$ is a  partial abelian ($\af_G\otimes \af'_G$)-extension of $R\otimes_R R\simeq R$. Thus,  by Theorem \ref{teofundpar2} $(S\otimes_R S')^{\alpha_{\delta G}}$ is a partial abelian $\af_{(G\times G)/\delta_G}$-extension of $R$. Hence, we have an operation on  $\Hh_{par}(G,R)$ given by
\[\lfloor S,\alpha_G\rfloor\ast_{par} \lfloor S',\alpha'_G\rfloor=\lfloor (S\otimes_R S')^{\alpha_{\delta G}}, \af_{(G\times G)/\delta_G}  \rfloor, \]
where $\lfloor L,\theta\rfloor$ denotes the  class of $G$-isomorphism of a partial abelian $\theta$-extension $L$ of $R$.
We will show that $\ast_{par}$ is a well-defined associative operation on  $T_{par}(G,R)$. In fact, we have the following.

\smallbreak

\noindent {\bf Theorem \ref{invpar}.} Let $G$ be a finite abelian group and $R$ a commutative algebra. Then  $\big(T_{par}(G,R),\ast_{par}\big)$  is a commutative inverse semigroup.
\smallbreak

We shall establish a relation between  $\Hh_{par}(G,R)$ and a Harrison group. Indeed, consider  $\lfloor S,\alpha_G\rfloor\in  \Hh_{par}(G,R)$ and let $(T,\beta_G)$ be the globalization of $(S,\alpha_G)$. By Theorem 3.3 of \cite{DFP},  $T$ is a (global) $G$-extension of $A=T^{\beta_G}$ and consequently its equivalence class $[T,\beta_G]$ is an element of $\Hh(G,A)$. Using Theorem \ref{pro10}, we will prove that the association $\lfloor S,\alpha\rfloor\mapsto [T,\beta]$ is a well-defined semigroup homomorphism and we have the following result.

\smallbreak

\noindent {\bf Theorem \ref{rel-semigrupo-grupo}.} 	Let $R$ be an algebra,  $\pi:\Hh_{par}(G,R)\to\Hh(G,A)$ the semigroup homomorphism defined by $\pi\big(\lfloor S,\af\rfloor\big)=[T,\beta]$ and $\rho=\ker\pi$.  Then, the canonical map induced by $\pi$ from $\Hh_{par}(G,R)/\rho$ to $\Hh(G,A)$ is an isomorphism of groups.

\smallbreak

Commutative inverse semigroups are strong semilattices of abelian groups. In particular, if $\I$ is a commutative inverse semigroup  then $\I=\bigcup_{e\in E(\I)} \I_{e}$ where $E(\I)$ is the semillatice of the idempotents of $\I$ and $\I_e$ is a group, for all $e\in E(\I)$. In the last  section of this paper we study the idempotents of  $\Hh_{par}(G,R)$.\smallbreak

The paper is organized as follows. The basic notions and results that we used throughout the paper are presented in Section \ref{notions}. In Section \ref{quo}, we will prove Theorem \ref{teofundpar2}. Following ideas from \cite{DFP}, in Section \ref{enve} we study partial Galois extensions which are $G$-isomorphic.
In Section 5, we prove that $\Hh_{par}(G,R)$ is a commutative inverse semigroup. Some aspects of the structure of $\Hh_{par}(G,R)$  are dealt with in Section 6.

\subsection*{Conventions}\label{subsec:conv}
Throughout the paper, $k$ will denote an associative and commutative ring with identity element. All algebras are $k$-algebras and they are considered associative with identity element. Each algebra homomorphism is unitary, that is, it preserves identity elements. Extensions of algebras have  the same identity element. If $S$ and $T$ are extensions of an algebra $R$ then $S\otimes T$ means $S\otimes_R T$. As usual, the annihilator of a central element $x$ of an algebra $T$ will be denoted by $\anula(x)$, i.~e. $\anula(x)=\{t\in T\,:\,tx=0\}$. Similarly, if $T$ is an algebra extension of $R$ and $x$ is a central element of $T$ then $\anula_R(x):=\{r\in R\,:\,rx=0\}$. We will denote the identity element of a group $G$ by $1$.

\section{Preliminaries}\label{notions}

In this section we present the background that will be used in the paper. The references used here are \cite{DE} and \cite{DFP}.

\subsection{Partial action of groups}\label{subsec-partial action} A \textit{partial action} of a group $G$ on a $k$-algebra $S$ is a family of pairs  $\alpha_G=(S_g,\alpha_g)_{g\in G}$
that satisfies:
\begin{enumerate}[\hspace{.35cm}(P1)]
	    \item for each $g\in G$, $S_g$ is an ideal of $S$ and $\af_{g}:S_{g\m}\to S_g$ is a $k$-algebra isomorphism,\vspace{.08cm}
		\item $S_1=S$ and $\alpha_1=\id_S$, \vspace{.08cm}
		\item $\alpha_g(S_{g^{-1}}\cap S_h)=S_g\cap S_{gh}$, for all $g,h \in G$,\vspace{.08cm}
		\item $\alpha_g \circ \alpha_h(x)=\alpha_{gh}(x)$, for all  $ x\in \af_{h\m}(S_h\cap S_{g\m})$ and $g,h \in G$.
\end{enumerate}
The partial action $\af_G$ is called {\it unital} if every ideal $S_g$ is unital, that is, there exists a central idempotent $1_g$ in $S$ such that $S_g=S1_g$.
Notice that the conditions (P3) and (P4) imply that $\alpha_{gh}$  is an extension of $\alpha_g \circ \alpha_h$, for every $g,h\in G$.
For a subgroup $H$ of $G$,  the partial action $\af_H$ of $H$ on $S$ is obtained by restriction of $\af_G$, i.~e. $\af_H=(S_h,\af_h)_{h\in H}$. Sometimes, we will write  $(S,\alpha_G)$ to denote a partial action of $G$ on $S$. \smallbreak

A partial action $\alpha_G=(S_g,\alpha_g)_{g\in G}$ is said {\it global} if $S_g=S$, for all $g\in G$. Global actions of a group $G$ on a $k$-algebra $T$ induce, by restriction, partial actions on any ideal $S$ of $T$. Indeed, given a global action  $\bt_G=(T_g,\beta_g)_{g\in G}$ of $G$ on $T$ we consider the ideals $S_g=S\cap \beta_g(S)$ of $S$ and the $k$-algebra isomorphisms $\af_g=\beta_g|_{S_{g\m}}$. Then $\af_G=(S_g,\af_g)_{g\in G}$ is a partial action of $G$ on $S$. We are interested in partial actions that are globalizable, that is, those that can be obtained as restriction of global actions. \smallbreak

A {\it globalization}  of a partial action  $\alpha_G=(S_g,\alpha_g)_{g\in G}$ of a group $G$ on a $k$-algebra $S$ is a global action $\bt_G=(T_g,\beta_g)_{g\in G}$  of $G$ on a $k$-algebra $T$  with a monomorphism $\varphi:S\to T$ of $k$-algebras such that:
\begin{enumerate}[\hspace{.35cm}(G1)]
	\item $\varphi(S)$ is an ideal of $T$, \vspace{.08cm}
	\item $\varphi(S_g)=\varphi(S) \cap \beta_g(\varphi(S))$, for all $g\in G$,\vspace{.08cm}
	\item $\beta_g (\varphi(x))=\varphi(\alpha_g(x))$, for all $x \in S_{g^{-1}}$,\vspace{.08cm}
	\item $T=\sum_{g \in G}\beta_g(\varphi(S))$.
\end{enumerate}
If $\af_G$ admits a globalization we say that $\af$ is {\it globalizable} and, in this case, its globalization is
unique, up to isomorphism. Also, Theorem 4.5 of \cite{DE} implies that  unital partial actions admit {globalization}. More details related to globalization can be seen in \cite{DE}.

\vd

From now on we assume that  $G$ is a finite group, $\af_G=(S_g,\alpha_g)_{g\in G}$ is a unital partial action of $G$ on a $k$-algebra $S$ such that $S_g=S1_g$, for all $g\in G$, and  $\beta_G=(T_g,\beta_g)_{g\in G}$ is a global action of $G$ on a $k$-algebra $T$ which is the globalization of  $\af_G$.  In order to simplify the notation, we assume that the injective morphism $\varphi$ in the definition of globalization is the inclusion map from $S$ to $T$, that is, $S$ is an ideal of $T$. Notice that $1_S$ is a central  idempotent of $T$ and $S=T1_S$. Moreover, it was proved in \cite[p.79]{DFP} that
\begin{align}\label{form1:partial-action}
&1_g=\bt_g(1_S)1_S,\quad \alpha_g(s1_{g\m})=\beta_g(s)1_S,\quad\af_g(\af_h (s1_{h\m})1_{g\m})=\af_{gh}(s1_{(gh)\m})1_g,&
\end{align}
for all $g,h\in G$ and $s\in S$. Let $H=\{h_1=1, h_2, \dots, h_m\}$ be a subgroup  of $G$. It was defined in \cite[p.79]{DFP} the map $\psi_H:T\to T$  by
\begin{equation} \label{fih}
\psi_H(t)=\sum_{1\leq l \leq m}\sum_{i_1< \cdots < i_l}(-1)^{l+1}\beta_{h_{i_1}}(1_S)\beta_{h_{i_2}}(1_S)\cdots \beta_{h_{i_l}}(1_S)\beta_{h_{i_l}}(t),\quad t\in T.
\end{equation}
This map can be rewritten as
\begin{equation*}
\psi_H(t)=\sum\limits_{i=1}^m\beta_{h_i}(t)e_i,\quad\text{for all}\,\, t\in T,
\end{equation*}
where $e_i$'s are the following idempotents of $T$:
\begin{equation}\label{losei}
e_1=1_S, \quad e_i=(1_T-1_S)(1_T-\beta_{h_2}(1_S))\cdots (1_T-\beta_{h_{i-1}}(1_S))\beta_{h_i}(1_S),\quad 2\leq i\leq m.
\end{equation}
Since $1_S$ is a central idempotent of $T$, given $2\leq i\leq m$ we have
\begin{align*}\bt_g(e_i)1_S&=(1_T-\bt_g(1_S))(1_T-\beta_{gh_2}(1_S))\cdots (1_T-\beta_{gh_{i-1}}(1_S))\beta_{gh_i}(1_S)1_S\\
&=(1_S-\bt_g(1_S)1_S)(1_S-\beta_{gh_2}(1_S)1_S)\cdots (1_S-\beta_{gh_{i-1}}(1_S)1_S)\beta_{gh_i}(1_S)1_S\\
&\stackrel{\eqref{form1:partial-action}}=(1_S-1_g)(1_S-1_{gh_2})\cdots (1_S-1_{gh_{i-1}})1_{gh_i}.
\end{align*}
Thus
\begin{equation}\label{btgei}
\bt_g(e_1)1_S=1_g\,\,\, \,\text{and}\,\,\,\, \bt_g(e_i)1_S=\prod\limits_{j=2}^i(1_S-1_{gh_{j-1}})1_{gh_i}, \quad 2\leq i\leq m.
\end{equation}

We recall from \cite{DFP} that  $S^{\alpha_G}:=\{s\in S :\alpha_g(s1_{g^{-1}})=s1_g,\,\text{for all}\,\, g\in G\}$ is called the {\it subalgebra of invariants} of $S$ under $\af_G$.
 If $\alpha_G$ is global then $S^{\alpha_G}$ is the classical subalgebra of invariants, i.~e. $S^{\alpha_G}=S^{G}=\{x\in S: \af_g(s)=s,\,\text{for all}\,\, g\in G\}$. \smallbreak

We shall denote by $e_H$ the image of $1_S$ by $\psi_H$, that is,
\begin{align*}
e_H=\psi_H(1_S).
\end{align*}
The element $e_H\in T$ will be useful in  Section \ref{quo}. Some properties of the map $\psi_H$ and of the element $e_H$ are given in the next.

\begin{prop}\label{pro3.1}
Let $H=\{h_1=1, h_2, \dots, h_m\}$ be a subgroup  of $G$ and  $\psi_H:T\to T$ the map defined in \eqref{fih}. Then:
	\begin{enumerate}[\rm (i)]
		\item $\psi_H$ is a $k$-algebra homomorphism,\vspace{0.1cm}
		\item $\psi_H$ is left and right $T^{\beta_H}$-linear map,\vspace{0.1cm}
		\item the restriction $\psi_H|_S$ to $S$ is injective,\vspace{0.1cm}
		\item $e_H$ is a central idempotent of $T$,\vspace{0.1cm}
		\item $\psi_H(S^{\alpha_H})\subset T^{\beta_H}$,\vspace{0.1cm}
		\item the restriction of $\psi_H$ to $S^{\alpha_H}$ is a $k$-algebra isomorphism from $S^{\alpha_H}$ onto $T^{\beta_H}e_H$ whose inverse is the multiplication by $1_S$. In particular $T^{\beta_H}1_S=S^{\alpha_H}$.
	\end{enumerate}
\end{prop}

\begin{proof}
Notice that (i) is immediate because  each $\beta_g$, $g\in G$, is a $k$-algebra homomorphism and $\{e_i\,:\,1\leq i\leq m\}$ is a set of orthogonal idempotents of $T$. 
Item (ii) is clear. For (iii), take $s\in S$ and observe that
\[\psi_H(s)1_S=\sum_{i=1}^m\beta_{h_i}(s)e_i1_S=\beta_1(s)1_S=s\]
because $e_i1_S=0$, for all $2\leq i\leq m$. Thus, (iii) follows.  Since $\{e_i\,:\,1\leq i\leq m\}$ is a set of central orthogonal idempotents of $T$ and $\psi_H(1_S)=e_1+\cdots +e_m$,  item (iv) follows.
	
For (v), we need to show that $\beta_h(\psi_H(s))=\psi_H(s)$, for all $s\in S^{\alpha_H}$ and $h\in H$. Thus, it is enough to check that $\beta_h(\psi_H(s))$ gives us a permutation  of the terms that appear in the sum of  $\psi_H(s)$ given in \eqref{fih}. First, observe that
\begin{align*}
\beta_{h_i}(1_S)\beta_{h_j}(s)&=\beta_{h_j}\left(\beta_{h_j^{-1}h_i}(1_S)s\right)=\beta_{h_j}\left(\beta_{h_j^{-1}h_i}(1_S)s1_S\right)\\
                              &\overset{\mathclap{\eqref{form1:partial-action}}}{=}\beta_{h_j}(1_{h_j^{-1}h_i}s)\stackrel{(\star)}{=}\beta_{h_j}(\beta_{h_j^{-1}h_i}(s)1_S)\\
                              &=\beta_{h_i}(s)\beta_{h_j}(1_S)\stackrel{(\star\star)}{=}\beta_{h_j}(1_S)\beta_{h_i}(s),
\end{align*}	
where $(\star)$ follows because $\beta_h(s)1_S\stackrel{\eqref{form1:partial-action}}{=}\af_h(s1_{h^{-1}})=s1_h$ and $(\ast\ast)$ follows using  the fact that $\beta_h(1_S)$ is a central element of $T$, for all $h\in H$. Now, for  $1\leq l\leq n$, consider $i_1<\dots<i_l$ and note that
\begin{align*}
\beta_h(\beta_{h_{i_1}}(1_S)\cdots \beta_{h_{i_{l-1}}}(1_S)\beta_{h_{i_l}}(s))&=\beta_{h{h_{i_1}}}(1_S)\cdots \beta_{h{h_{i_{l-1}}}}(1_S)\beta_{h{h_{i_l}}}(s)\\
                                                                              &=\beta_{{h_{j_1}}}(1_S)\cdots \beta_{{h_{j_{l-1}}}}(1_S)\beta_{{h_{j_l}}}(s),
\end{align*}
where $h_{j_k}=hh_{i_k}$ for all $1\leq k\leq l$. Since each $\beta_{h}(1_S)$, $h\in H$, is a central element of $T$, we can rearrange $\beta_{h_{j_1}}(1_S),\ldots,\beta_{h_{j_{l-1}}}(1_S)$ such that  $j_1<\cdots <j_{l-1}$. If $j_{l-1}<j_l$, there is nothing to do. Otherwise, as we showed above, $\beta_{h_{l-1}}(1_S)\beta_{h_l}(s)=\beta_{h_{l}}(1_S)\beta_{h_{l-1}}(s)$. Thus, $\beta_h(\beta_{h_{i_1}}(1_S)\cdots \beta_{h_{i_{l-1}}}(1_S)\beta_{h_{i_l}}(s))$ appears in the sum of $\psi_H(s)$ and the result follows.

Finally  for (vi), notice that by (iv) and (v) we have $\psi_H(S^{\af_H})=\psi_H(S^{\af_H}1_S)\subset T^{\beta_H}e_H$. Hence, by (1), we conclude that  $\psi_H: S^{\alpha_H} \rightarrow T^{\beta_H}e_H$ is a $k$-algebra  homomorphism. Also, if $x\in T^{\beta_H}$ then
\begin{align*}
	\alpha_h((x1_S)1_{h^{-1}})&=\beta_h((x1_S)1_{h^{-1}})\stackrel{\eqref{form1:partial-action}}{=}\beta_h(x1_S\beta_{h^{-1}}(1_S))\\
	                          &=\beta_h(x)\beta_h(1_S)1_S\stackrel{\eqref{form1:partial-action}}{=}x1_h=(x1_S)1_h,
\end{align*}
for all $h\in H$. Hence,  $x1_S\in S^{\alpha_H}$. Thus,  $\varphi:T^{\beta_H}e_H\to S^{\alpha_H}$, given by $\varphi(a)=a1_S$ for all $a\in T^{\beta_H}e_H$, is a  well-defined $k$-algebra homomorphism because $e_H1_S=1_S$. Clearly, $\varphi(\psi_H(s))=s$
and $\psi_H(\varphi(xe_H))=xe_H$, for all $s\in S^{\alpha_H}$ and $x\in T^{\beta_H}$.
\end{proof}

\section{Partial actions of quotient groups}\label{quo}
From now on all algebras will be commutative. Consider an algebra $R$ and a unital partial action $(S,\af_G)$  of $G$ on an algebra $S$. According
\cite{DFP}, $S$ is an  $\af_G$-\textit{partial Galois extension} of $R$ (shortly, a partial $\af_G$-extension of $R$)  if $R\simeq S^{\alpha_G}$  (as $k$-algebras) and  there exist
 $m\in \mathbb{N}$ and elements $x_i,y_i\in S, 1\leq i\leq m$, such that
\begin{equation*}\label{G2}
\sum_{i=1}^mx_i\alpha_{g}(y_i1_{g^{-1}})=\delta_{1, g},\, \text{for each}\, g \in G.
\end{equation*}
The elements $x_i,y_i$  are called \textit{partial Galois coordinates} of $S$ over $R$. By abuse of notation, sometimes we will write $R=S^{\af_G}$, even when $R$ is some isomorphic copy of $S^{\alpha_G}$. \smallbreak

Let $S$ be a partial $\alpha_G$-extension of $S^{\alpha_G}$. From Theorem 5.1 of \cite{DFP} we have a correspondence between  subgroups of $G$ and certain subalgebras of $S$. Now we give an addendum of this result. \smallbreak

\begin{thm}\label{teofundpar2}
	Let $S$ be an algebra, $G$ a finite group, $\alpha_G$ a unital partial action of $G$ on $S$ such that $S$ is a partial  $\alpha_G$-extension of $S^{\alpha_G}$ and $H$ a normal subgroup of $G$. Then  $\af_G$ induces a unital partial action $\alpha_{G/H}$ of the quotient group $G/H$ on $S^{\alpha_H}$ such that $S^{\alpha_H}$ is  a partial $\alpha_{G/H}$-extension of $S^{\alpha_G}$ and $\big(S^{\af_H}\big)^{\af_{G/H}}=S^{\af_G}$.
\end{thm}

The proof of Theorem \ref{teofundpar2} will be obtained as a consequence of several results which we state and prove below.  Let $H$ be  a normal subgroup of $G$. By  \cite[Theorem 2.3]{CHR}, the global action  $\beta_G$ of $G$ on $T$ induces a global action $\beta_{G/H}$ of $G/H$ on $T^{\beta_H}$ in the following way:
$$\beta_{G/H}:G/H\rightarrow \text{Aut}(T^{\beta_H}),\quad gH\mapsto \beta_g|_{T^{\beta_H}}.$$ Also, if $T$ is a  $G$-extension of $T^{\beta_G}$, then $T^{\beta_H}$ is a  $G/H$-extension of $T^{\beta_G}=(T^{\beta_H})^{\beta_{G/H}}$.

On the other hand, by Proposition \ref{pro3.1} (iv), $T^{\beta_H}e_H$ is an ideal of $T^{\af_H}$. Then the action $\beta_{G/H}$ of $G/H$ on $T^{\beta_H}$ induces,  by restriction, a partial action $\gamma_{G/H}$ of $G/H$ on $T^{\beta_H}e_H,$  that is, $\gamma_{G/H}=(D_{gH}, \gamma_{gH})_{gH\in G/H}$ is given by:
\begin{align}
\label{par1-action}&D_{gH}=(T^{\beta_H}e_H)\cap \beta_g(T^{\beta_H}e_H)=T^{\beta_H}e_{gH},\,\,\,\text{where}\,\,\, e_{gH}:=e_H\beta_g(e_H),&\\[.3em]
&\gamma_{gH}=\beta_g|_{D_{g^{-1}H}}:D_{g^{-1}H}\to D_{gH},\,\,\, \text{for each}\,\,\, gH\in G/H.&
\label{par2-action}\end{align}

\begin{prop} \label{lem2}  The pair $(T^{\beta_H},\beta_{G/H})$ is the globalization of the partial action $\gamma_{G/H}$ given in \eqref{par1-action} and \eqref{par2-action}.
\end{prop}
\begin{proof} By  construction, (G1), (G2) and (G3) of  subsection 2.1 are satisfied. For (G4), it is enough to check that $T^{\beta_H}=\sum\limits_{i=1}^l\beta_{g_i}(T^{\beta_H}e_H),$ where  $\mathcal{T}=\{g_1=1, g_2,\ldots,g_l\}$ is a transversal of $H$ on $G.$ Let $H=\{h_1=1,h_2,\dots , h_m\},$  and write the elements of $G$ in the following order:
\begin{equation}\label{order-G}
G=\{1,h_2,\dots , h_m,g_2,g_2h_2,\dots , g_2h_m,\dots g_l, g_lh_2,\dots , g_lh_m\}.
\end{equation}
We claim that  $\bt_g(T^{\beta_H})=T^{\beta_H}$, for all $g\in G$. In fact, consider $a\in T^{\beta_H}$, $g\in G$ and $h\in H$. Since  $H$ is a normal subgroup of $G$, there is $h'\in H$ such that $hg=gh'$. Then \[\beta_h(\beta_g(a))=\beta_{hg}(a)=\beta_{gh'}(a)=\beta_g(\beta_{h'}(a))=\beta_g(a).\]  Thus  $I:=\sum\limits_{i=1}^l\beta_{g_i}(T^{\beta_H}e_H)=\sum\limits_{i=1}^lT^{\beta_H}\beta_{g_i}(e_H)$  is an ideal of $T^{\beta_H}$. It is enough to show that  $1_T\in I$.
Take $f_H:=\prod\limits_{j=1}^m(1_T-\bt_{h_j}(1_S))$. It is clear that $f_H\in T^{\beta_H}$ and  whence
$\bt_{g}(f_H)\in  T^{\beta_H}$, for all $g\in G$.
Denote by $e_{ij}\in T$ the idempotent, constructed in \eqref{losei}, associated to the $(mi+j)$-th element of $G$ by considering the order given in \eqref{order-G}. Explicitly
\[e_{ij}=\prod_{1\leq  k \leq i-1}\beta_{g_k}(f_H)(1_T-\beta_{g_ih_1}(1_S))\cdots(1_T-\beta_{g_ih_{j-1}}(1_S))\beta_{g_ih_j}(1_S).\]
Notice that $e_{1j}=(1_T-\beta_{h_1}(1_S))\cdots(1_T-\beta_{h_{j-1}}(1_S))\beta_{h_j}(1_S)$. Thus
\[e_{ij}=\prod_{1\leq  k \leq i-1}\beta_{g_k}(f_H)\beta_{g_i}(e_{1j}).\]
Since $e_H=\sum\limits_{j=1}^{m}e_{1j}$ it follows that $\sum\limits_{j=1}^{m}e_{ij}= \big(\prod_{1\leq  k \leq i-1}\bt_{g_k}(f_H)\big)\bt_{g_i}(e_{H})\in I$. By \cite[p.79]{DFP} we have that  $1_T=\sum\limits_{1\leq  i \leq l\atop 1\leq j \leq m}e_{ij}\in I$ and consequently $I=T^{\beta_H}$.
\end{proof}

By Proposition \ref{lem2} we can consider a linear map $\psi_{G/H}:T^{\beta_H}\to T^{\beta_H}$ in a similar way as in \eqref{fih}. Indeed, if $G/H=\{g_1H, \dots, g_mH\}$ then
\begin{equation}\label{pi-GH}
\psi_{G/H}(t)=\sum_{1\leq l \leq m}\sum_{i_1< \cdots < i_l}(-1)^{l+1}\beta_{g_{i_1}}(e_H)\beta_{g_{i_2}}(e_H)\cdots \beta_{g_{i_l}}(e_H)\beta_{g_{i_l}}(t),\quad t\in T^{\beta_H}.
\end{equation}

\begin{cor}\label{cor-action-quotient} Let $\gamma_{G/H}$ be the partial action of $G/H$ on $T^{\beta_H}e_H$ given by \eqref{par1-action} and \eqref{par2-action}, $\psi_{G/H}$ the map given by \eqref{pi-GH} and $R:= (T^{\beta_H}e_H)^{\gamma_{G/H}}$ the subalgebra of invariant elements of $T^{\beta_H}e_H$.  Then the following statements hold:
\begin{enumerate}[\rm (i)]
	\item  $\psi_{G/H}$ is a $T^{\beta_G}$-linear homomorphism of $k$-algebras whose restriction to $T^{\beta_H}e_H$ is injective and $\psi_{G/H}(e_H)=1_T=\psi_G(1_S)$. \smallbreak
	\item $\psi_{G/H}(R)\subset  (T^{\beta_H})^{{\beta_{G/H}}}=T^{\beta_G}$. \smallbreak
	\item The restriction $\psi_{G/H}$ on $R$ is a $k$-algebra isomorphism of $R$ onto $T^{\beta_G}=(T^{\beta_H})^{{\beta_{G/H}}}$ with inverse given by the multiplication by $e_H$. Particularly, $T^{\beta_G}e_H=R$.\smallbreak
	\item $T^{\beta_H}e_H$ is a  partial $\gamma_{G/H}$-extension of $R$ if and only if $T^{\beta_H}$ is a  $\beta_{G/H}$-extension of $T^{\beta_G}=(T^{\beta_H})^{\beta_{G/H}}$.
\end{enumerate}		
\end{cor}

\begin{proof}
	Items (i), (ii) and (iii) follow directly from Proposition \ref{pro3.1}, while item (iv) follows from Proposition \ref{lem2} and \cite[Theorem 3.3]{DFP}.
\end{proof}

We shall see that the partial action $\gamma_{G/H}$ on $T^{\beta_H}e_H$ induces a partial action $\alpha_{G/H}$ of $G/H$ on $S^{\alpha_H}$ via  multiplication by $1_S$.
We set
\begin{align}
\label{idempgi}&\tilde{1}_{gH}:=e_{gH}1_S=e_H\beta_{g}(e_H)1_S=\beta_{g}(e_H)1_S,\\[.2em]
\label{dprima} & \tilde{D}_{gH}:=D_{gH}1_S=T^{\beta_H}e_{gH}1_S=S^{\alpha_H}\tilde{1}_{gH},\\[.2em]
\label{alfaprima}& \alpha_{gH}:=(m_{1_S}\circ \gamma_{gH}\circ \psi_H)|_{\tilde{D}_{g^{-1}H}},
\end{align}
where $m_{1_S}:T^{\beta_H}e_H\to S^{\alpha_H}$ is the multiplication by $1_S$. Notice that, for every $s\in S^{\af_H}$ and $gH\in G/H$,
\begin{align}\label{action-quotient}
\af_{gH}(s\tilde{1}_{g\m H})=\beta_g(\psi_H(s))\tilde{1}_{gH}=\beta_g(\psi_H(s))1_S.
\end{align}

\begin{thm}\label{pro1.2.6.} Let $H$ be a normal subgroup of $G$  and $\alpha_{G/H}:=\big(\tilde{D}_{gH},\alpha_{gH}\big)_{gH\in G/H}$,  with
$\tilde{D}_{gH}$ and $\alpha_{gH}$ given respectively by \eqref{dprima} and \eqref{alfaprima}. Then the following statements hold:
\begin{enumerate}[\rm (i)]
		\item $\alpha_{G/H}$ is a partial action of $G/H$ on $S^{\alpha_H}$.\smallbreak
		\item   $(S^{\alpha_H})^{\alpha_{G/H}}=S^{\alpha_G}$. \smallbreak
		\item $(T^{\beta_H}, \beta_{G/H})$ is the globalization of $\alpha_{G/H}$.\smallbreak
		\item $S^{\alpha_H}$ is a partial $\af_{G/H}$-extension of $S^{\alpha_G}$  if and only if $T^{\beta_H}$ is a partial $\beta_{G/H}$-extension of $T^{\beta_G}$.
	\end{enumerate}
\end{thm}

\begin{proof} 	(i) Given $g\in G$, it is clear from \eqref{idempgi}  that $\tilde{1}_{gH}$ is a central idempotent of $T$. We will check that $\tilde{1}_{gH}\in S^{\alpha_H}$. Since $H$ is a normal subgroup of $G$, for each $h\in H$, there exists $h'\in H$ such that $hg=gh'$. Then
	\begin{align*}
	\af_h(\tilde{1}_{gH}1_{h\m})=\af_h(\beta_{g}(e_H)1_{h\m})=\bt_{hg}(e_H)1_{h}
	=\bt_{gh'}(e_H)1_{h}.
	\end{align*}
By Proposition \ref{pro3.1} (v) and (vi), $e_H=\psi(1_S)\in T^{\beta_H}$. Thus
\begin{align*}
\af_h(\tilde{1}_{gH}1_{h\m})=\bt_{g}(\beta_{h'}(e_H))1_{h}=\beta_{g}(e_H)1_S1_h=\tilde{1}_{gH}1_h.
\end{align*}
Hence $\tilde{D}_{gH}$ is an ideal of $S^{\alpha_H}$. Further, by Proposition \ref{pro3.1},
	\begin{equation}\label{impsi}\psi_H(S^{\alpha_H}\tilde{1}_{g\m H})=\psi_H(S^{\alpha_H})\beta_{g\m }(e_H)e_H=T^{\beta_H}e_{g\m H}.\end{equation}
Then
\begin{align*}
	\alpha_{gH}(\tilde{D}_{g\m H})\stackrel{\eqref{impsi}}=m_{1_S}\circ \gamma_{gH}(T^{\beta_H}e_{g\m H})=T^{\beta_H}e_{gH}1_S=S^{\alpha_H}\beta_{g}(e_H)1_S\stackrel{\eqref{dprima}}=\tilde{D}_{gH}.
	\end{align*}
Thus $\alpha_{gH}:\tilde{D}_{g\m H}\to \tilde{D}_{g H}$ is a well-defined algebra isomorphism. Hence the condition (P1) of  subsection \ref{subsec-partial action} is satisfied while  (P2) is clear. For (P3), consider $gH, l H\in G/H$ and note that
	\begin{align*}
		\alpha_{gH}(\tilde{D}_{g\m H}\cap \tilde{D}_{lH})&=(m_{1_S}\circ \gamma_{gH}\circ \psi_H)(S^{\alpha_H}\tilde{1}_{g\m H}\tilde{1}_{lH})\\
		&\overset{\mathclap{\eqref{impsi}}}{=}(m_{1_S}\circ \gamma_{gH})(T^{\beta_H}e_{g\m H}e_{l H})\stackrel{(*)}{=} m_{1_S}(T^{\beta_H}e_{gH}e_{lg H})\\
		&=T^{\beta_H}e_{gH}1_ST^{\beta_H}e_{lg H}1_S=\tilde{D}_{g_H}\cap \tilde{D}_{glH}.
		\end{align*}
The equality in ($\ast$) follows from (P3) because $\gamma_{G/H}$ is a partial action.
Finally, to check (P4), take  $x\in \tilde{D}_{l\m H}\cap \tilde{D}_{(gl)\m H}$. Then,
		\begin{align*}
		\alpha_{gH}\circ \alpha_{lH}(x)&=(m_{1_S}\circ \gamma_{gH}\circ\psi_H)\circ(m_{1_S}\circ \gamma_{hH}\circ\psi_H)(x)\\
		&=(m_{1_S}\circ (\gamma_{gH}\circ \gamma_{lH})\circ \psi_H)(x)\\
		&\overset{\mathclap{\eqref{impsi}}}{=}(m_{1_S}\circ \gamma_{glH}\circ \psi_H)(x)\\
		&=\alpha_{glH}(x).
		\end{align*}

\noindent (ii)
Note that
$\tilde{1}_{g\m H}\in S^{\alpha_H}$. Since $\psi_H$ is  left $T^{\beta_H}$-linear and $\beta_{g\m}(e_H)\in T^{\beta_H}$, it follows that $\psi_H(\tilde{1}_{g\m H})=\beta_{g\m}(e_H)e_H=e_{g\m H}$ (the idempotents $e_{gH}$ are defined in \eqref{par1-action}). This implies $\psi_H((S^{\alpha_H})^{\alpha_{G/H}})=(\psi_H(S^{\alpha_H}))^{\gamma_{G/H}}.$ Indeed, let $s\in S^{\alpha_H}$ such that $ \psi_H(s)\in(\psi_H(S^{\alpha_H}))^{\gamma_{G/H}}$.
Then
	\begin{align*}
	\af_{gH}(s\tilde{1}_{g\m H})&\overset{\mathclap{\eqref{action-quotient}}}{=}\beta_g(\psi_H(s))\tilde{1}_{g H}=\beta_{g}(\psi_H(s)e_{g\m H})\tilde{1}_{g H}\\
	                            &=\gamma_{gH}(\psi_H(s)e_{g\m H})\tilde{1}_{g H}=\psi_H(s)e_{g H}1_S=s\tilde{1}_{g H}.
	\end{align*}
In an analogous way one shows the other inclusion.  Finally we have
	\begin{align*}(S^{\alpha_H})^{\alpha_{G/H}}&\overset{\mathclap{(*)}}{=}\psi_H((S^{\alpha_H})^{\alpha_{G/H}})1_S=(\psi_H(S^{\alpha_H}))^{\gamma_{G/H}}1_S\\
	&=(T^{\beta_H}e_H)^{\gamma_{G/H}}1_S\stackrel{(**)}=T^{\beta_G}e_H1_S\\
	&= T^{\beta_G}1_S=S^{\alpha_G},\end{align*}
where ($\ast$) follows from  Proposition \ref{pro3.1} (vii) and ($\ast\ast$) follows from  Corollary \ref{cor-action-quotient} (iii).	\smallbreak

\noindent (iii) By Proposition $\ref{pro3.1}$, there exists an algebra monomorphism $\psi_H: S^{\alpha_H}\to  T^{\beta_H}$ such that $\psi_H(S^{\alpha_H})=T^{\beta_H}e_H$ is an ideal of $T^{\beta_H}$, which implies (G1).   Furthermore, it follows from \eqref{impsi} that $\psi_H(\tilde{D}_{g_H})=T^{\beta_H}e_{gH}=\psi_H(S^{\alpha_H})\cap \beta_{g}(\psi_H(S^{\alpha_H}))$ and whence (G2) is satisfied.
By \eqref{action-quotient}, for each $s\in S^{\alpha_H}$,
\begin{align*}
\psi_H(\alpha_{gH}(s\tilde{1}_{g\m H}))&=\psi_H(\beta_g(\psi_H(s))1_S)=\beta_g(\psi_H(s))\psi_H(1_S)\\
                                       &=\beta_g(\psi_H(s))e_H=\beta_g(\psi_H(s)\tilde{1}_{g\m H}).
\end{align*}
Thus (G3) is verified. Using  Proposition \ref{lem2} we obtain
$$T^{\beta_H}=\sum\limits_{gH\in G/H}\beta_{gH}(T^{\beta_H}e_H)=\sum\limits_{gH\in G/H} \beta_{gH}(\psi_H(S^{\alpha_H})),$$
and whence (G4) is proved. \smallbreak

\noindent (iv) It follows from (iii) and  \cite[Theorem 3.3]{DFP}.
\end{proof}

\noindent {\it Proof of Theorem \ref{teofundpar2}.}  By Theorem \ref{pro1.2.6.}, there exists a partial action $\alpha_{G/H}$ of the quotient group $G/H$ on $S^{\alpha_H}$ induced by  $\alpha_G$. Observe that by Theorem 2.3 of \cite{CHR}, $T^{\beta_H}$ is a $\beta_{G/H}$-Galois extension of $T^{\beta_G}$. Hence,
Theorem \ref{pro1.2.6.} (iv) implies that $S^{\alpha_H}$ is an $\af_{G/H}$-partial Galois extension of $S^{\alpha}$.   \qed
\smallbreak	\smallbreak

\begin{rem}\label{obs-global-case}
	{\rm Suppose that $\alpha_G$ is a global action of $G$ on $S$ and $H$ is a normal subgroup of $G$. Then the $k$-algebra isomorphism $\psi_H$ defined in  subsection \ref{subsec-partial action}  is equal to $\id_{S}$. Thus, $e_H=1_S$. Hence the partial action $\alpha_{G/H}$ constructed in \eqref{dprima} and \eqref{alfaprima} is the usual global action of $G/H$ on $S^H$ given by Theorem 2.2 of \cite{CHR}.}
\end{rem}

In the next proposition we characterize when the partial action $\af_{G/H}$ of $G/H$ on $S^{\af_H}$ given by \eqref{dprima} and \eqref{alfaprima} is global.

\begin{prop}\label{prop-induc-global} Let $\af_{G/H}$ be the partial action of $G/H$ on $S^{\af_H}$ given by \eqref{dprima} and \eqref{alfaprima} and assume that $G/H=\{g_1H=H,\,g_2H,\ldots,g_mH\}$. Then $\af_{G/H}$ is a global action of $G/H$ on $S^{\af_H}$ if and only if  $\{\beta_{g_i}(1_S)\,:\,i=1,\ldots,m\}\subset \anula(1_T-e_H)$.
\end{prop}
\begin{proof}
Since $\tilde{1}_{g_i\m H}=\beta_{g_i\m}(e_H)1_S$, we have that $\af_{G/H}$ is a global action on $S^{\af_H}$ if and only if $\beta_{g_i\m}(e_H)1_S=1_S$ for all $1\leq i\leq m$. Applying $\beta_{g_i}$ in both sides of the last identity we obtain that $e_H\beta_{g_i}(1_S)=\beta_{g_i}(1_S)$ and consequently the result follows.
\end{proof}

An immediate consequence of the above proposition is the following.

\begin{cor}\label{cor-indu-global}
If $e_H=1_T$ then the partial action $\af_{G/H}$ of $G/H$ on $S^{\af_H}$ is global.
\end{cor}

The next result describes explicitly the idempotents and the isomorphisms given in  \eqref{idempgi} and \eqref{alfaprima} respectively. It will be useful in  subsection \ref{harsemi}.

\begin{prop}\label{desc-idem-iso} Let $ H=\{h_1=1, h_2, \dots, h_m\}$ be a normal subgroup of $G$, $\tilde 1_{gH}$ given by \eqref{idempgi}, $\tilde{D}_{gH}$ given by \eqref{dprima} and $\alpha_{gH}$ given by \eqref{alfaprima}. Then
\begin{equation}\label{1gh}
\tilde 1_{gH}=1_g+\sum_{i=2}^m\prod\limits_{j=2}^i(1_S-1_{gh_{j-1}})1_{gh_i},
\end{equation}
and
\begin{equation}\label{alphapprima}\alpha_{gH}(x)=\af_{g}(x1_{g\m})+ \sum_{i=2}^m\prod\limits_{j=1}^{i-1}(1_S-1_{gh_{j}})\af_{gh_i}(x1_{(gh_i)\m}),\quad x\in \tilde{D}_{g\m H}.
\end{equation}
\end{prop}
\begin{proof}
Since $\tilde 1_{gH}=\beta_{g}(e_H)1_S$ and $e_H=\psi_H(1_S)$ we have	
\begin{align*}\tilde 1_{gH}&=\sum_{i=1}^m\beta_{gh_i}(1_S)\beta_{g}(e_i)1_S=\sum_{i=1}^m\beta_{gh_i}(1_S)1_S\beta_{g}(e_i)1_S=\sum_{i=1}^m1_{gh_i}\beta_{g}(e_i)1_S
\\&\stackrel{\eqref{btgei}}=1_g+\sum_{i=2}^m1_{gh_i}\prod\limits_{j=2}^i(1_S-1_{gh_{j-1}})1_{gh_i}=
1_g+\sum_{i=2}^m\prod\limits_{j=2}^i(1_S-1_{gh_{j-1}})1_{gh_i}.
\end{align*}
Let $g\in G$ and $x\in \tilde D_{g^{-1}H}$. By \eqref{alfaprima},
\begin{align*} \alpha_{gH}(x)&=(m_{1_S}\circ \gamma_{gH}\circ \psi_H)(x)
=\sum_{i=1}^m\beta_{gh_i}(x)\beta_{g}(e_i)1_S
\\&=\sum_{i=1}^m\af_{gh_i}(x1_{(gh_i)\m})\beta_{g}(e_i)1_S
\\&=\af_{g}(x1_{g\m})+ \sum_{i=2}^m\af_{gh_i}(x1_{(gh_i)\m})\prod\limits_{j=2}^i(1_S-1_{gh_{j-1}}).
\\&=\af_{g}(x1_{g\m})+ \sum_{i=2}^m\prod\limits_{j=1}^{i-1}(1_S-1_{gh_{j}})\af_{gh_i}(x1_{(gh_i)\m}).
\end{align*}
\end{proof}

We end this section with examples that illustrate the constructions given above.

\begin{exe}\label{ex0}{\rm
	Let  $R$ be  a commutative $k$-algebra and $S:=Re_1\oplus Re_2 \oplus Re_3$, where $e_1, e_2, e_3$ are non-zero orthogonal idempotents and  $e_1+e_2+e_3=1_S$. We consider the partial action $\alpha_G$ of the cyclic group of $G=C_4=\langle g\,:\,g^4=1 \rangle$ of order $4$ on $S$ given in Example 6.1 of \cite{DFP}. The ideals are
	\[S_g=Re_1\oplus Re_2,\qquad  S_{g^2}=Re_1\oplus Re_3,\qquad S_{g^3}=Re_2\oplus Re_3,\]
	and the isomorphisms are
	\[\af_g(ae_2+be_3)=ae_1+be_2,\quad \af_{g^2}(ae_1+be_3)=ae_3+be_1, \quad \af_{g^3}(ae_1+be_2)=ae_2+be_3, \]
for all $a,b\in R$. 
Given the subgroup $H=C_2=\{1,g^2\}$ of $C_4$, note that the quotient group is $G/H=\{H,gH\}\simeq C_2$ and $S^{\alpha_{H}}=R(e_1+e_3)\oplus Re_2$.
 It  is clear that  $(T,\beta_G)$ is the globalization of  $\alpha_G$, where $T=S\oplus Re_4$ and $\beta_g(e_j)=e_{j-1\,(\text{mod}\,4)}$. From \eqref{fih}  it follows that
\[\psi_{H}(ae_1+be_2+ce_3+de_4)=ae_1+b(e_2+e_4)+ce_3,\quad \text{for all } a,b,c,d\in R.\]
Hence, $e_{H}=\psi_{H}(1_S)=1_T$ and whence $\tilde{1}_{gH}=\beta_g(1_T)1_S=1_S$. Thus $\tilde{D}_{gH}=S^{\af_H}$ and using \eqref{action-quotient} we conclude that $\af_{gH}(a(e_1+e_3)+be_2)=b(e_1+e_3)+ae_2$. In this case, ${G/H}$ acts globally on $S^{\alpha_H}$. }
\end{exe}

\begin{exe}\label{ec6R}\cite[Example 4.2]{PRS}\label{ex01} {\rm Let  $R$ be a commutative $k$-algebra and $T:=\oplus_{i=1}^6Re_i$ where $\{e_i\}_{i=1}^6$ is a set of  non-zero orthogonal idempotents whose sum is $1_T$. Consider the action $\beta_G$ of the cyclic group $G=C_6=\langle g\,:\,g^6=1\rangle$ of order $6$ on $T$ given by
\[\bt_{g^j}\Big(\sum_{i=1}^6a_ie_i\Big)=\sum_{i=1}^6a_ie_{i+j\, ({\rm mod\,\, 6})},\quad a_i\in R. \]
Let $S:=Re_1\oplus Re_3\oplus Re_6$ and  $\alpha_G=(S_g,\alpha_g)_{g\in G}$ be the induced partial action of  $\beta_G$ on $S$. Explicitly
\[S_g=Re_1,\,\,\,\,\, S_{g^2}=Re_3,\,\, \,\,\, S_{g^3}=Re_3\oplus Re_6,\,\, \,\,\,S_{g^4}=Re_1,\,\, \,\,\, S_{g^5}=Re_6,  \]
and $\alpha_{g^j}=\beta_{g^j}:S_{g^{6-j}}\to S_{g^j}$, for all $1\leq j\leq 5$. By construction,  $(T,\beta_G)$ is the globalization of  $\alpha_G$.  Given the subgroup $H=C_2=\{1,g^3\}$  of $C_6$, the quotient group is $G/H=\{H,gH,g^2H\}\simeq C_3$ and $S^{\alpha_{H}}=Re_1\oplus R(e_3+e_6)$.
 From \eqref{fih} it follows that
\[\psi_{H}\Big(\sum\limits_{i=1}^6a_ie_i\Big)=a_1(e_1+e_4)+a_3e_3+a_6e_6,\quad \text{for all } a_i\in R, \quad 1\leq i \leq  6.\]
Thus $e_H=\psi_{H}(1_S)=e_1+e_3+e_4+e_6$ and consequently
\[\tilde{1}_{H}=1_S, \quad \tilde{1}_{gH}=\beta_g(e_H)1_S=e_1, \quad \tilde{1}_{g^2H}=\beta_{g^2}(e_H)1_S=e_3 + e_6.\]
Hence the partial action $\af_{G/H}=(\tilde{D}_g,\af_{gH})_{gH\in G/H}$ of $G/H$ on $S^{\alpha_H}$ is given by
\[\tilde{D}_{H}=S^{\alpha_H}, \quad  \tilde{D}_{gH}=Re_1,\quad  \tilde{D}_{g^2H}=R(e_3+e_6),\]
and
\[\af_H=\id_{S^{\af_H}},\quad \af_{gH}(a(e_3+ a_6))=ae_1,\quad \af_{g^2H}(ae_1)=a(e_3+e_6),\quad a\in R.\]}
\end{exe}


\section{On $G$-isomorphic partial Galois extensions }\label{enve}

In this section we introduce the notion of $G$-isomorphic partial $\af_G$-extensions and we investigate its relation with $G$-isomorphic  $G$-extensions  given in \cite{H}. \smallbreak

Let $\alpha_G=(S_g,\af_g)_{g\in G}$ and $\alpha_G'=(S'_g,\af'_g)_{g\in G}$ be partial actions of $G$ on the commutative $k$-algebras $S$ and $S'$ respectively. Assume
 that  $S$ (resp. $S'$) is a partial $\af_G$-extension (resp. $\af'_G$-extension) of $R\simeq S^{\alpha_G}$ (resp. $R\simeq (S')^{\af'_G}$).
We  say  that $(S, \alpha_G)$ and $(S',\alpha'_G)$ are \emph{partially $G$-isomorphic} if there exists a $k$-algebra isomorphism $f: S\rightarrow S'$ which is $R$-linear and satisfies
\begin{align}\label{def:par-G-iso}
&f(S_g)\subseteq S'_g\,\,\,\text { and }\,\,\, (f \circ \alpha_g)|_{S_{g^{-1}}} = (\alpha'_g \circ f)|_{S_{g^{-1}}},\quad \text{for all}\, g\in G.&
\end{align}
In this case, we will denote $(S,\alpha_G)\overset{par}{\sim} (S',\alpha'_G)$. Clearly $\overset{par}{\sim}$ is an equivalence relation on the set of partial  $\af_G$-extensions of $R$. The equivalence class of  $(S,\alpha_G)$ will be denoted by  $\lfloor S,\alpha_G\rfloor$. \smallbreak

\begin{rem}\label{obs-igual}{\rm Notice that if $f:S\to S'$ is a $k$-algebra isomorphism such that $f(S_g)\subset S'_g$ then $f(S_g)= S'_g$. Indeed, since $f$ is surjective, given $a'_g\in S'_g$ there exists $a\in S$ such that $f(a)=a'_g$. Also, $f(1_{g})=1'_{g}$. Thus, $a'_g=f(a)1'_g=f(a1_g)$ which implies that $f(S_g)=S'_g$. }
	
\end{rem}

In order to prove the next result we recall from \cite{DFP} the definition of  the trace map. Let  $\alpha_G=(S_g,\af_g)_{g\in G}$ be a partial action of $G$ on $S$ and  $R=S^{\alpha_G}$. The trace map $tr_{\alpha_G}:S\to S$ is defined by
\[tr_{\alpha_G}(s)=\sum\limits_{g \in G}\alpha_{g}(s1_{g^{-1}}), \quad s\in S.\]
By \cite[Lemma 2.1]{DFP},  $tr_{\af_G}(S)\subset R$ and $tr_{\af_G}$ is a (left and right) $R$-linear map.

\begin{prop}\label{pro9}
Let $(S,\alpha_G)$ (resp. $(S',\alpha'_G)$) be a partial $\af_G$-extension (resp. $\af'_G$-extension) of $R$. 
Then  $(S, \alpha_G)$ and $(S',\alpha'_G)$  are partially \emph{$G$-isomorphic} if and only if there exists a $k$-algebra homomorphism $f: S\rightarrow S'$ which is $R$-linear and satisfies \eqref{def:par-G-iso}.
\end{prop}
\begin{proof} Let  $\alpha_G=(S_g,\af_g)_{g\in G}$ and  $\alpha'_G=(S'_g,\af'_g)_{g\in G}$ be partial actions of $G$ on $S$ and $S'$ respectively. Assume that $f:S\to S'$ is a $k$-algebra homomorphism which is $R$-linear and satisfies \eqref{def:par-G-iso}. Since $S$ is a partial $\af_G$-extension of $R$, consider $x_i, y_i\in S$, $1\leq i \leq n$, the partial Galois coordinates of $S$ over $R$. For each $s'\in S'$, we have
	\begin{align*}
	f\Big(\sum_{i=1}^nx_itr_{\alpha'}(f(y_i1_{g^{-1}})s')\Big)&=\sum_{i=1}^n\sum_{g \in G}f(x_i)\alpha'_g(f(y_i1_{g^{-1}}))\alpha'_g(s'1'_{g^{-1}})\\
	&=\sum_{i=1}^n\sum_{g \in G}f(x_i)(\alpha'_g\circ f)(y_i1_{g^{-1}})\alpha'_g(s'1'_{g^{-1}})\\
	&=\sum_{i=1}^n\sum_{g \in G}f(x_i)(f\circ \alpha_{g})(y_i1_{g^{-1}})\alpha'_g(s'1'_{g^{-1}})\\
	&=\sum_{g \in G}f\Big(\sum_{i=1}^nx_i\alpha_{g}(y_i1_{g^{-1}})\Big)\alpha'_g(s'1'_{g^{-1}})\\
	&=f(1_S)s'=1_{S'}s'=s',
	\end{align*}
which implies that $f$ is surjective. Suppose that $s\in S$ and $f(s)=0$. Then
	\begin{align*}
	f(\alpha_{g}(y_is1_{g^{-1}}))&=(f\circ \alpha_{g})(y_is1_{g^{-1}})=(\alpha'_g\circ f)(y_is1_{g^{-1}})\\
	&=\alpha_g'(f(y_i)f(s)f(1_{g^{-1}}))=0,
	\end{align*}
	for all $g \in G$. Since $tr_{\af_G}(S)\subset R$, we have $tr_{\alpha_G}(y_is)1_{S'}=f(tr_{\alpha_G}(y_is)1_S)$. Thus,  $tr_{\alpha_G}(y_is)=0$, for all $1\leq i\leq n$.
	Hence,
	\begin{align*}
	0&=\sum_{i=1}^nx_i tr_{\alpha_G}(y_is)=\sum_{g \in G}\Big(\sum_{i=1}^nx_i\alpha_{g}(y_i1_{g^{-1}})\Big)\alpha_{g}(s1_{g^{-1}})\\
	&=\sum_{g \in G}\delta_{1,g}\alpha_{g}(s1_{g^{-1}})=s,
	\end{align*}
	which proves that $f$ is injective.
\end{proof}

Let  $(T,{\beta_G})$ and $(T',{\beta'_G})$ be $G$-extensions of $A\simeq T^{\beta_G}\simeq (T')^{\beta'_G}$.  According  to  \cite{H},
 $(T,{\beta_G})$ and $(T',{\beta'_G})$ are $G$-isomorphic if there is an $A$-linear homomorphism of $k$-algebras $f:T\to T'$ such that $f\circ\beta_g=\beta'_g\circ f$, for all $g\in G$. In this case, $f$ is in fact an isomorphism; more details can be seen in \cite{H}.  \smallbreak

For our purposes, the concept of (global) $G$-isomorphism  for globalization of partial actions needs one more condition than the classical notion. Let $(S,{\alpha_G})$ and $(S',{\alpha'_G})$ be partial actions of $G$ such that $S$ and $S'$ are, respectively, partial  $\af_G$ and $\af'_G$-extensions of $R$. Let  $(T,{\beta_G})$ and $(T',{\beta'_G})$  be their respective globalizations. It follows from Theorem 3.3 of \cite{DFP} that  $T$ and $T'$ are $G$-extensions of $A\simeq T^{\beta_G}\simeq (T')^{\beta'_G}$.  In this case,
we say that $\big((T,{\beta_G}), (S, {\alpha_G})\big)$ and $\big((T',{\beta'_G}),(S', {\alpha'_G})\big)$ are {\it $G$-isomorphic} if there exists  an $A$-linear  $G$-isomorphism  $f:T\to T'$ such that $f(1_{S})=1_{S'}$. We will write $\big((T,{\beta_G}),(S, {\alpha_G})\big)\sim\big((T',{\beta'_G}),(S', {\alpha'_G})\big)$ to indicate that  $\big((T,{\beta_G}),(S, {\alpha_G})\big)$ and $\big((T',{\beta'_G}), (S', {\alpha'_G})\big)$ are $G$-isomorphic. Clearly $\sim$ is an equivalence relation.
\smallbreak

Let $(S,\af_G)$ be a partial action of $G$ on $S$ and $H$ a normal subgroup of $G$. As in section \ref{quo}, here $(S^{\alpha_H},\alpha_{G/H})$ denotes the partial action of $G/H$ on $S^{\alpha_{H}}$ given by \eqref{dprima} and \eqref{alfaprima}.

\begin{thm}\label{pro10} Let  $(S,{\alpha_G})$ and   $(S',{\alpha'_G})$ be partial actions of $G$ and  $(T,{\beta_G})$  and $(T',{\beta'_G})$ their respective globalizations. 
Assume that $S$ and $S'$  are, respectively,  partial $\af_G$ and $\af'_G$-extensions of $R$. Then the following statements are equivalent:
\begin{enumerate}[\rm (i)]
\item $\big( (T,{\beta_G}),( S, {\alpha_G})\big)\sim\big( (T',{\beta'_G}),(S', {\alpha'_G})\big)$,   \smallbreak

\item $(S^{\alpha_H},\alpha_{G/H})\overset{par}{\sim}((S')^{\alpha'_H},\alpha'_{G/H})$, for all normal subgroup $H$ of $G$,   \smallbreak

\item $(S, \alpha_G)\overset{par}{\sim}(S',\alpha'_G)$.
\end{enumerate}	
\end{thm}

\begin{proof}
(i) $\Rightarrow$ (ii) We suppose that $H=\{h_1=1,h_2,\ldots,h_m\}$, $A\simeq T^{ \beta_G}\simeq (T')^{ \beta_G}$ and $R\simeq S^{ \alpha_G}\simeq (S')^{ \alpha'_G}$. Consider an $A$-linear $k$-algebra isomorphism $f: T\rightarrow T'$ such that
\begin{align}\label{G-iso3}
	&f(1_S)=1_{S'},& 	&f \circ \beta_g=\beta'_g\circ f,\quad \text{for all}\,\,\, g \in G.&
\end{align}
It is easy to check that $f|_{T^{ \beta_H}}:T^{\beta_H}\to (T')^{ \beta'_H}$ is  an $A$-linear $k$-algebra  isomorphism.
By  Proposition \ref{lem2}, $(T^{\beta_H},\beta_{G/H})$ is the globalization of $(T^{\beta_H}e_H,\gamma_{G/H})$, where $\gamma_{G/H}$ is given by \eqref{par1-action} and \eqref{par2-action}. Observe that $e_H=\psi_H(1_S)=\sum_{j=1}^m\beta_{h_j}(1_S)e_j$, where the family $\{e_i\}_{1\leq i\leq m}$ is given by \eqref{losei}. Similarly, we have that $((T')^{\beta'_H},\beta'_{G/H})$ is the globalization of $((T')^{\beta'_H}e'_H,\gamma'_{G/H})$, where $e'_H=\sum_{j=1}^m\beta'_{h_j}(1_{S'})e'_j$ and
$$e'_1=1_{S'}\,\,\,\,\text{and}\,\,\,\, e'_j=(1_{T'}-1_{S'})\cdots(1_{T'}-\beta'_{h_{j-1}}(1_{S'}))\beta'_{h_j}(1_{S'}),\quad 1\leq j\leq m.$$
It follows from (\ref{G-iso3}) that $f(e_j)=e'_j$, for all $1\leq j\leq m$.	Thus, $f(T^{\beta_H}e_H)=(T')^{\beta'_H}e'_H$. In particular, $f(e_H)=e'_H$. \smallbreak

We recall the partial actions $\alpha_{G/H}$ and $\alpha'_{G/H}$ of $G/H$ on $S^{\af_H}$ and $(S')^{\af'_{H}}$ respectively. They are given by \eqref{dprima} and \eqref{alfaprima}, that is, the ideals are
$\tilde{D}_{gH}=T^{\beta_H}e_H\beta_{g}(e_H)1_S$ and  $\tilde{D'}_{gH}=(T')^{\beta'_H}e'_H\beta'_{g}(e'_H)1_{S'}$
and the partial isomorphisms are
\[\alpha_{gH}=(m_{1_S}\circ\gamma_{gH}\circ\psi_H)|_{\tilde{D}_{g\m H}},\quad\quad \alpha'_{gH}=(m_{1_{S'}}\circ\gamma'_{gH}\circ\psi'_H)|_{\tilde{D'}_{g\m H}}.\]
Consider the map $\varphi=m_{1_{S'}}\circ f\circ\psi_H:S^{\af_H}\to (S')^{\af'_H}$. By Proposition \ref{pro3.1}, $\psi_H(S^{\alpha_{H}})\subset T^{\beta_H}$. Since $f(T^{\beta_H})=(T')^{\beta'_H}$ and $(T')^{\beta'_H}1_{S'}=(S')^{\alpha'_{H}}$, it follows that $\varphi$ is well-defined. Clearly, $\varphi$ is a $k$-algebra  homomorphism.  We will check that $\varphi$ is $R$-linear. By  item (vi) of Proposition \ref{pro3.1} we obtain $S^{\alpha_G}=T^{\beta_G}1_S\simeq (T')^{\beta'_G}1_{S'}=(S')^{\af'_G}$. Hence, given $a\in T^{\beta_G}$ we have that $a1_S\in S^{\alpha_G}$. Moreover there exists a unique $a'\in (T')^{\beta'_G}$ such that $a'1_{S'}\in (S')^{\af'_G}$.  Then, for $s\in S^{\af_H}$,
	\begin{align*}
	\varphi((a1_S)s)&=(m_{1_{S'}} \circ f \circ \psi_H)((a1_S)s)\stackrel{(*)}{=}(m_{1_{S'}} \circ f)(a \psi_H(1_S)\psi_H(s))\\[.2em]
	&=(m_{1_{S'}}\circ f)(a e_H \psi_H(s))=1_{S'}(a'f(e_H)f(\psi_H(s)))\\[.2em]
	&=1_{S'}(a'e'_Hf(\psi_H(s)))=(a'1_{S'})(1_{S'}e'_H)f(\psi_H(s))\\[.2em]
	&=(a'1_{S'})(1_{S'}f(\psi_H(s))=(a'1_{S'})\varphi(s).
	\end{align*}
The equality $(\ast)$ follows because $\psi_H$ is $A$-linear by Proposition \ref{pro3.1} (ii). Now, we check that $\varphi$ satisfies \eqref{def:par-G-iso}. For each $g\in G$, we have
	\begin{align*}
	\varphi(\tilde{D}_{gH})&\overset{\mathclap{\eqref{impsi}}}{=}m_{1_{S'}} \circ f(T^{{\beta_H}}e_{gH})=(m_{1_{S'}} \circ f)(T^{\beta_H}\beta_{g}(e_H)e_H)\\[.3em]
	&=1_{S'}T'^{\beta'_H}f(\beta_{g}(e_H))f(e_H)=1_{S'}T'^{\beta'_H}\beta'_{g}(f(e_H))e'_H\\[.3em]
	&=T'^{\beta'_H}e'_H\beta'_{g}(e'_H)1_{S'}=\tilde{D'}_{gH}.
	\end{align*}
Also, if $x\in \tilde{D}_{g\m H}$ then
	\begin{align*}
	(\varphi \circ \alpha_{gH})(x)&
	=(m_{1_{S'}}\circ (f \circ \gamma_{gH}) \circ \psi_H)(x)\\
	&\overset{\mathclap{\eqref{G-iso3}}}{=}(m_{1_{S'}}\circ (\gamma'_{gH} \circ f) \circ \psi_H)(x)\\[.3em]
	&=(m_{1_{S'}}\circ \gamma'_{gH} \circ \id_{(T')^He'_H}\circ f \circ \psi_H)(x)\\
	&\overset{\mathclap{(*)}}{=}((m_{1_{S'}}\circ \gamma'_{gH} \circ (\psi'_H \circ m_{1_{S'}})\circ f \circ \psi_H))(x)\\[.3em]
	&=((m_{1_{S'}}\circ \gamma'_{gH} \circ \psi'_H) \circ (m_{1_{S'}}\circ f \circ \psi_H))(x)\\[.3em]
	&=(\alpha'_{gH}\circ \varphi)(x).
	\end{align*}
The equality $(*)$ follows from Proposition \ref{pro3.1} (vi) because $(T')^{{\beta'_H}}e'_H1_{S'}=(S')^{\af'_H}$ and $\psi'_H\big((S')^{\af'_H}\big)=(T')^{{\beta'_H}}e'_H$. 	
By  Proposition \ref{pro9}, $(S^{\alpha_H},\alpha_{G/H})$ and $((S')^{\alpha'_H},\alpha'_{G/H})$ are partially $G/H$-isomorphic.\smallbreak

\noindent It is clear that (ii) $\Rightarrow$ (iii) follows by taking the trivial subgroup $H=\{1\}$ of $G$. Finally, for  (iii) $\Rightarrow$ (i), we consider a $k$-algebra isomorphism $f: S\rightarrow S'$  satisfying \eqref{def:par-G-iso} and an injective homomorphism of $k$-algebras  $\varphi': S'\to T'$ satisfying (G1)-(G4) of $\S\,$\ref{subsec-partial action}. We claim  that $(T',\beta'_G)$ is a globalization of $(S,\alpha_G)$. In fact, consider the $k$-algebra monomorphism $\phi=\varphi' \circ f :S\to T'$.  We shall check that $\phi$ also satisfies the conditions (G1)-(G4). Observe that (G1) is immediate. Let $g\in G$. Using Remark \ref{obs-igual}, we have that
\begin{align*}\phi(S_g)&=\varphi'(f(S_g))=\varphi'(S'_g)\stackrel{\rm{(G2)}}{=}\varphi'(S')\cap \beta'_g(\varphi'(S'))\\[.1em]
&=(\varphi' \circ f)(S)\cap (\beta'_g\circ \varphi'\circ f)(S)=\phi(S)\cap\beta'_g(\phi(S)),\end{align*}
which implies that $\phi$ satisfies (G2).
For (G3), notice that	
\begin{align*}
(\phi\circ \alpha_g)(x)&=(\varphi'\circ f\circ \alpha_g )(x)
=(\varphi'\circ \alpha'_g \circ f)(x)
\stackrel{\rm{(G3)}}{=}(\beta'_g \circ \varphi' \circ f)(x)
=(\beta'_g \circ \phi)(x),
\end{align*}
for all $g\in G$ and $x \in S_{g^{-1}}$. Finally,
\[T'=\sum_{g \in G}\beta'_g(\varphi'(S'))=\sum_{g \in G}\beta'_g((\varphi'\circ f)(S))=\sum_{g \in G}\beta'_g(\phi(S)).\]
Thus $(T',\beta'_G)$ is a globalization of $(S,\alpha_G)$. Since $(T,\beta_G)$ is also a globalization of $(S,\alpha_G)$, by Theorem 4.5 of \cite{DE}, the global actions $(T',\beta'_G)$ and  $(T,\beta_G)$ are equivalent. Particularly, $A\simeq T^{{\beta_G}}\simeq (T')^{{\beta'_G}}$.  Also, it follows from the proof of Theorem 4.5 of \cite{DE} that  $\Psi:T\to T'$ given by $\Psi\big(\sum_{g\in G}\beta_{g}(\varphi(s))\big)=\sum_{g\in G}\beta'_{g}(\phi(s))$ is an $A$-linear algebra isomorphism such that $\Psi\circ\beta_g=\beta'_g\circ\Psi$, for all $g\in G$, and  $\Psi(1_S)=1_{S'}$.
\end{proof}

\begin{exe}{\rm Let $G=\langle g\,: \, g^4=1 \rangle$ be the cyclic group of order $4$. Consider an algebra $R$ and $T:=Re_1\oplus Re_2\oplus Re_3\oplus Re_4$, where $\{e_i\}_{1\leq i\leq 4}$ is a set of orthogonal idempotents whose sum is $1_T$. The group $G$ acts on $T$ by  $\beta_{g^i}(e_j)=e_{j-i\,({\rm mod}\, 4)}.$ Let $e=e_1+e_2+ e_3\in T$, $S=Te$ and ${\alpha_G}$ be the induced partial action of  $G$ on $S$, i.~e. ${\alpha_G}=(S_g,\alpha_g)_{g\in G}$ with $S_{g^i}=S\cap \beta_{g^i}(S)$ and $\alpha_{g^i}=\beta_{g^i}|_{S_{g^{4-i}}}$. By Example 6.1 of \cite{DFP}, $(S,{\alpha_G})$ is a partial { $\af_G$-}extension of $R$. Notice that $(T,{\beta_G})$ is $G$-isomorphic to $(T,{\beta_G})$ but $\big((T,{\beta_G}), (T,{\beta_G})\big)$ is not $G$-isomorphic to $\big((T,{ \beta_G}),(S,{\alpha_G})\big)$.}
\end{exe}

\section{The commutative inverse semigroup {$\Hh_{par}(G,R)$}}\label{sec-semigroup}

From now on, $G$  will denote a finite abelian group. Given a commutative $k$-algebra $R$, ${ \Hh_{par}}(G,R)$ denotes the set of equivalence classes of \emph{partial abelian {$\af_G$-}extensions} of $R$, that is, the elements of ${ \Hh_{par}}(G,R)$ are the classes $\lfloor S,{\af_G}\rfloor$, where  $(S,{\alpha_G})$ is a partial action of $G$ on $S$, $R=S^{{\af_G}}$  and $R\subset S$ is a partial abelian $\af_G$-extension. In this section we define an operation on  $\Hh_{par}(G,R)$ which turns it into a commutative inverse semigroup.  First we recall from \cite{H} the classical construction of the Harrison group.

\subsection{{ The} Harrison group}\label{harglo}
Let  $A$ be a commutative $k$-algebra and ${ \Hh(G,A)}$ the set of the equivalence  classes of $G$-isomorphic abelian extensions of $A$. Given a Galois extension $(B, {\beta_G})$ of $A$ (which means that ${\beta_G}$ is a global action of $G$ on $B$, $A\simeq B^{{\beta_G}}$ and $A\subset B$ is an abelian $\beta_G$-extension), we will denote its corresponding equivalence class by $\left[B,{\bt_G} \right]$. It was shown in \cite{H} that ${ \Hh(G,A)}$ is an abelian group and now we recall this construction. \smallbreak

Given $\left[B,{\bt_G}\right], \left[B', {\bt'_G}\right]\in { \Hh(G,A)}$, the tensor product $B\otimes_A B'$ is an abelian $\bt_{G\times G}$-extension of $A$. By Theorem 2.2 of \cite{CHR},  $(B\otimes_A B')^{\delta G}$ is  an abelian $\bt_{G\simeq (G\times G)/\delta G}$-extension of $A$, where $\delta G=\{(g,g\m): g\in G\}.$
The group $G$ acts on $(B\otimes_A B')^{\delta G}$ via the group homomorphism $\gamma_{G}:G\to \operatorname{Aut}\big((B\otimes_A B')^{\delta G}\big)$ given by
\begin{equation}\label{prod-h-global}
\gamma_g\big(\sum_{i=1}^l b_i\otimes b'_i\big)=\sum_{i=1}^l\bt_g(b_i)\otimes b'_i=\sum_{i=1}^l b_i\otimes \bt'_g(b'_i),\quad b_i\in B,\,\,b'_i\in B'.
\end{equation}
Consequently
\begin{equation}\label{harprod}\left[B,{\bt_G}\right]\ast \left[B', {\bt'_G}\right]=\big[(B\otimes_A B')^{\delta G}, \gamma_{(G\times G)/\delta G}\big]\end{equation}
defines an associative and commutative operation  on ${\Hh(G,A)}$. The identity element of ${\Hh(G,A)}$ is the class $\left[E_G(A),{\rho_G}\right]$ which is defined in the following way. Consider symbols $\{e_g\}_{g\in G}$ and the free $A$-module  $E_G(A)=\oplus_{g\in G} Ae_g$ with basis  $\{e_g\}_{g\in G}$. The product on $E_G(A)$ is defined by $(ae_g)(a'e_h)=\delta_{g,h}(aa')e_g$, for all $a,a'\in A$ and $g,h\in G$. Then $E_G(A)$ is a $k$-algebra with unity $1_{E_{G}(A)}=\sum_{g\in G}e_g$. The action ${\rho_G}$  of $G$ on $E_G(A)$ is given by  $\rho_g(ae_{h})=ae_{gh}$, for all $a\in A$ and $g,h\in G$.
The inverse of $\left[B,{\bt_G}\right]$ is the equivalence class $\left[B,{\beta_G^{-1}}\right]$, where the action ${\beta_G^{-1}}$ of $G$ on $B$ is defined by $\bt^{-1}_g(b):=\bt_{g^{-1}}(b)$, for all $g\in G$ and $b \in B$.
\smallbreak

The group ${\Hh(G,A)}$ is called {\it Harrison group} and its study reduces to the case $G$ cyclic; see \cite{H} for more details.

\subsection{The construction of {$\Hh_{par}(G,R)$}}\label{harsemi}
We recall that a semigroup $\I$ is regular if for each $x\in \I$, there exists $x^\ast\in \I$ such that $xx^\ast x=x$ and $x^\ast xx^\ast=x^\ast$. A regular semigroup $\I$ is said an {\it inverse semigroup} if the idempotents of $\I$ commute with each other. Then, commutative regular semigroups are inverse semigroups. It is known that a  semigroup $\I$ is  inverse if and only if for each $x\in \I$ there exists a unique $x^\ast\in \I$ such that $xx^\ast x=x$  and $x^\ast xx^\ast=x^\ast$. The element $x^\ast$ is called the {\it  inverse} of $x$. The set of idempotents of $\I$ is
$E(\I)=\{x^*x\,:\, x\in \I\}$. The set $E(\I)$ admits a partial order via: $e\leq f$ if $ef=e$, where $e,f\in E(\I)$.
It is well-known that $E(\I)$ is a meet semilattice with respect to the partial order $\leq$. We suggest \cite{HO} and \cite{L}  for more details on regular and inverse semigroups.
\smallbreak

In order to construct a commutative inverse semigroup, denote by  ${\Hh_{par}(G,R)}$ the set of all equivalence class $\lfloor S,{\af_G}\rfloor$, where $(S,{\af_G})$ is a partial action of $G$ on $S$, $R\simeq S^{{\af_G}}$ and $R\subset S$ is a partial abelian $\af_G$-extension.  Consider $\lfloor S,{\alpha_G}\rfloor,\lfloor S',{\alpha'_G}\rfloor\in {\Hh_{par}(G,R)}$. By Proposition 2.9 of \cite{DPP},  $S\otimes S'$ is  a partial abelian $\af_{G\times G}$-extension of $R\otimes R\simeq R$, where $\af_{G\times G}:=\alpha_G\otimes \alpha'_G$.
We define on ${\Hh_{par}(G,R)}$ the operation $\ast_{par}$ by
\begin{align}\label{harsemprod}
\lfloor S,{\alpha_G}\rfloor \ast_{par}\lfloor S',{\alpha'_G}\rfloor
&=\lfloor(S\otimes S')^{{\af_{\delta G}}}, {\af_{(G\times G)/\delta G}}\rfloor,
\end{align} where $\delta G=\{(g,g^{-1})\,:\,g\in G\}$ and  ${\alpha_{(G\times G)/\delta G}}$ is given by
\eqref{dprima} and \eqref{alfaprima}.\smallbreak\smallbreak

Now we will prove the main result of this paper.

\begin{thm}\label{invpar}
	${(\Hh_{par}(G,R),\ast_{par})}$ is a commutative inverse semigroup.
\end{thm}

\begin{proof}
Suppose that $(S_1,{\alpha_{G,1}})\overset{par}{\sim} (S'_1,{\alpha'_{G,1}})$ and $(S_2,{\alpha_{G,2}}) \overset{par}{\sim} (S'_2,{\alpha'_{G,2}})$. Then, it is clear that
$(S_1\otimes S_2,\,{\alpha_{G,1}}\otimes {\alpha_{G,2}})\overset{par}{\sim} (S'_1\otimes S'_2,\,{\alpha'_{G,1}}\otimes {\alpha'_{G,2}})$.
Using the implication (iii) $\Rightarrow $ (ii) of Theorem \ref{pro10}  for the subgroup $\delta G$ of $G\times G$, we conclude that $\ast_{par}$ is well-defined.
Observe that from Theorem \ref{teofundpar2}  it follows that ${\Hh_{par}G,R})$ is closed under  such a product. It is clear that $\ast_{par}$ is an associative and commutative operation.
Thus it is enough to  check that ${\Hh_{par}(G,R)}$ is regular.
Let $\lfloor S,{\alpha_G}\rfloor$ be an element of ${\Hh_{par}(G,R)}$.  Assume that $(T,{\beta_G})$ is the globalization of $(S,{\alpha_G})$ and $T^{{\beta_G}}=A$. Consider the partial action
${\alpha^{\ast}_G}=(S^{\ast}_g,\af^{\ast}_g)_{g\in G}$  of $G$ on $S$ defined by $S^\ast_g:=S_{g\m}$ and $\af^\ast_g:=\af_{g\m}$, for all $g\in \G$. Thus, the global action $(T,{\beta\m_G})$ of $G$ on $T$ given by $\beta\m_g:=\beta_{g\m}$ for each $g\in G$, is the globalization of $(S,{\af^\ast_G})$. Since $[T,{\beta_G}]\ast [T,{\beta_G}]^{-1}=[E_G(A),{\rho_G}]$ in ${ T(G,A)}$, it follows that $[ T,{\beta_G}]\ast [ T,{\beta\m_G}]\ast [ T,{\beta_G}]=[T,{\beta_G}]$. \smallbreak

Denote $T':=(T\otimes_A T)^{\delta G}$ and $T'':=(T'\otimes_A T)^{\delta G}$. Then
\begin{align*}
\big([T,\beta_G] \ast [T,\beta\m_G]\big)\ast[T,\beta_G]&=[ T', \gamma'_G]\ast[T,\beta_G]\\
&=[T'', \gamma''_G]
\end{align*}
where $\gamma'_{G},\gamma''_{G}$ are given by \eqref{prod-h-global} and $\delta G=\{(g,g^{-1})\,:\,g\in G\}$. Also, the map
\begin{align*}
&f:T''\to T,& &f\big((x\otimes y)\otimes z\big)=xyz,&
\end{align*}
is a  $k$-algebra homomorphism which is $A$-linear and $f\big((1_S\otimes 1_S)\otimes 1_S\big)=1_S$.
Note that $\beta_g\circ f=f\circ \gamma''_g $, for all $g\in G$. Indeed, consider $t=(t_1\ot t_2)\ot t_3\in T''$. Then
\begin{align*}
\gamma''_g(t)&\overset{\mathclap{\eqref{prod-h-global}}}{=}(t_1\otimes t_2)\otimes \bt_g(t_3)=
(t_1\otimes \beta_{g\m}(\bt_g(t_2)))\otimes \bt_g(t_3)\\
&\stackrel{(*)}=(\bt_g(t_1)\otimes \bt_g(t_2))\otimes \bt_g(t_3),
\end{align*}
where $(\ast)$ follows because $t_1\otimes t_2\in T'$. Thus, $(\beta_g\circ f)(t)=(f\circ \gamma''_g)(t)$. Since the maps $f,\beta_g,\gamma''_g$ are $A$-linear, we conclude that $\beta_g\circ f=f\circ \gamma''_g $. Then, by Proposition \ref{pro9}, $f$ is an isomorphism.
By Theorem \ref{pro1.2.6.},  $(T'', \gamma''_{G})$ is the globalization of $(S'',\alpha''_{G})$ where
\[S''=(S'\otimes S)^{\af''_{\delta G}},\quad S'=(S\otimes S)^{\af'_{\delta G}},\quad \af''_{G}=\alpha'_{G'}\otimes{\alpha_G},\quad {\af'_{G}=\alpha_G\otimes\af^{\ast}_G},\] and $G'=(G\times G)/\delta G\simeq G$.
Thus   $\big((T'',\gamma''_{G}), (S'',\alpha''_{G}\big)\sim\big((T,\beta_{G}),(S,\alpha_{G})\big)$. It follows from  Theorem \ref{pro10} that $ (S'',\alpha''_{G})\overset{par}{\sim} (S,\alpha_{G})$, that is, $\lfloor S,\alpha_{G}\rfloor \ast_{par}\lfloor S,\af^{\ast}_{G}\rfloor\ast_{par}\lfloor S,\alpha_{G}\rfloor=\lfloor S,\alpha_{G}\rfloor$.
In a similar way one shows that $\lfloor S,\af^{\ast}_{G}\rfloor\ast_{par}\lfloor S,\alpha_{G}\rfloor\ast_{par}\lfloor S,\af^{\ast}_{G}\rfloor=\lfloor S,\af^{\ast}_{G}\rfloor$. Consequently $\lfloor S,\af^{\ast}_{G}\rfloor=\lfloor S,\alpha_{G}\rfloor^{\ast}$ and  ${\Hh_{par}(G,R)}$ is a regular semigroup.
\end{proof}

In the next example we calculate explicitly an idempotent of $\Sh(G,R)$.

\begin{exe}{\rm
 Let $G=\langle g\,:\, g^4=1\rangle$ be the cyclic group of order $4$.
Consider an algebra $R$ and $S:=Re_1\oplus Re_2$, where $e_1, e_2$ are non-zero orthogonal idempotents whose sum is one. Here we calculate the idempotent of ${\Hh_{par}(G,R)}$ associated to an easy partial action of $G$ on $S$. We define the partial action ${\theta_G}$ of $G$ on $S$ by taking $S_1=S$, $S_{g}=Re_2$, $S_{g^2}=\{0\}$, $S_{g^3}=Re_1$, and setting  $\theta_1=\id_{S}$, $\theta_{g}(re_1)=re_2$, {$\theta_{g^2}=0$}, $\theta_{g^3}(re_2)=re_1$, $r\in R$.
Clearly $S^{{\theta_G}}=R$. Also, $x_1=y_1=e_1$ and $x_2=y_2=e_2$ are partial Galois coordinates of $S$ over $R$.
We will calculate the idempotent $\lfloor S,{\theta_G}\rfloor^\ast \ast_{par}  \lfloor S,{\theta_G}\rfloor$. Denote by $e_{ij}:=e_i\otimes e_j$ and $e_{i0}=e_i\otimes 1_S$. It is straightforward to check that
$$\tilde{S}:=(S\otimes S)^{{{\delta G}}}=R(e_{11}+e_{22})\oplus Re_{12}\oplus Re_{21},$$
where ${\theta^\ast_G}$ is the partial action of $G$ on $S$ defined by $\theta^{\ast}_g=\theta_{g\m}$, for all $g\in G$.
Now we calculate the idempotents $1_{(l,1)\delta_G},$  with $l\in G$. Let  $h_i=g^{i-1}$, $1\leq i\leq 4$. By \eqref{1gh},
\[1_{(l,1) \delta G}=1_{l\m}\otimes 1_S+\sum_{i=2}^4\prod\limits_{j=1}^{i-1}(1_S\otimes 1_{S}-1_{(lh_{j})\m}\otimes 1_{h^{-1}_{j}})(1_{(lh_i)\m}\otimes 1_{h^{-1}_i}).\]
Consequently
\begin{align*}
&\tilde 1_{(1,1)\delta G}=1_S\ot 1_S,&  &\tilde 1_{(g,1)\delta G}=e_{10}+e_{22},&
&\tilde 1_{(g^2,1)\delta G}=e_{12}+e_{21},& &\tilde 1_{(g^3,1)\delta G}=e_{20}+e_{11}.&
\end{align*}
It is immediate to verify that
\begin{align*}
\tilde D_{(g,1)\delta G}=R(e_{11}&+e_{22})\oplus Re_{12},\qquad \tilde D_{(g^2,1)\delta G}=Re_{12}\oplus Re_{21},\\[.2em]
&\tilde D_{(g^3,1)\delta G}=R(e_{11}+e_{22})\oplus Re_{21}.
\end{align*}
It follows from (\ref{alphapprima}) that $\lfloor S,\theta_{G}\rfloor^\ast\ast_{par}  \lfloor S,\theta_{G}\rfloor=\lfloor \tilde{S},\tilde{\theta_{G}} \rfloor$, where
\begin{align*}
&\,\,\,\tilde{\theta} _{g}\big(r(e_{11}+e_{22})\oplus se_{21}\big)=r(e_{11}+e_{22})+se_{12},\\[.2em]
\tilde{\theta} _{g^2}&\big(re_{12}+se_{21}\big)=se_{12}+re_{21}, \quad  \tilde{\theta} _{g^3}=\tilde{\theta}\m_{g},\quad r,s\in R .
\end{align*}}
\end{exe}

\subsection{A {group isomorphism}} Let $\varphi:\I\to \mathcal{J}$ be a homomorphism of semigroups. We recall from  subsection 2.3 of \cite{L} that the kernel of $\varphi$ is the equivalence relation on $\I$ defined by
\[\ker \varphi:=\{(x,y)\in \I\times \I\,:\,\varphi(x)=\varphi(y)\}.\] In fact, $\rho:=\ker\varphi$ is a congruence on $\I$. Denote by $\rho(a)$ an equivalence class on $S$ and by $\I/\rho$ the set of all equivalence classes. Thus $\I/\rho$ has the following natural structure of semigroup: $\rho(a)\rho(b)=\rho(ab)$, for all $a,b\in S$. This operation is well-defined because $\rho$ is a congruence. If $\varphi$ is surjective then $\I/\rho\simeq \mathcal{J}$ (as semigroups).

Let $R$ be an algebra and $\lfloor S,{\af_G}\rfloor\in {\Hh_{par}(G, R)}$. Consider a globalization $(T,{\beta_G})$ of $(S,{\alpha_G})$. By Theorem 3.3 of \cite{DFP}, $T$ is  an abelian $\bt_G$-extension of $T^{{\beta_G}}\simeq A$, that is, $[T,{\beta_G}]\in {T(G,A)}$. Moreover, if $(S,{\af_G}),\,(S',{\af'_G})$  are, respectively, partial abelian $\af_G$ and $\af'_G$-extensions of $R$ and $(T,{\beta_G}),\,(T',{\beta'_G})$ are  their respective globalizations, then we have $T^{{\beta_G}}\simeq (T')^{{\beta'_G}}$. Hence, we can define the following map
\begin{align}\label{hom-pi}
\pi\,:\,{\Hh_{par}(G, R)}\to  {\Hh(G,A)},\qquad\pi \big(\lfloor S,{\af_G}\rfloor\big)=[T,{\beta_G}].
\end{align}
By Theorem \ref{pro10}, $\pi$ is well-defined. It is clear that $\pi$ is a surjective homomorphism of semigroups. Hence, we have the following result.

\begin{thm}\label{rel-semigrupo-grupo}
Let $R$ be an algebra and  $\rho=\ker\pi$, where $\pi$ is given by \eqref{hom-pi}. Then \[\overline{\pi}\,:\,{T_{par}(G,R)}/\rho\to T(G,A),\quad \overline{\pi}\big(\rho\big(\lfloor S,{\alpha_G}\rfloor\big)\big)=[T,{\beta_G}],\]	
is a  group isomorphism.
\end{thm}

\subsection{Reduction to cyclic groups}
As in the classical case, the study of partial Galois extensions of finite abelian groups reduces to cyclic groups. In fact, let $G$ be a finite abelian group and assume that $G=G_1\times G_2\times \cdots \times G_n$, where $G_1, G_2, \dots, G_n$ are  cyclic groups. For each $1\leq i\leq n$,  let $(S_i,{\af_{G_i}})$ be a partial action of $G_i$ on $S_i$, $R=S_i^{{\af_{G_i}}}$ and assume that $R\subset S_i$ is a partial abelian $\af_{G_i}$-extension.  Consider $S:=S_1\otimes S_2\otimes \cdots \otimes S_n$ and ${\af_G:=\af_{G_1}\otimes \af_{G_2}\otimes\cdots \otimes \af_{G_n}}$. By Proposition 2.9 of \cite{DPP},
$(S,{\af_G})$ is a partial action and $S$ is  a partial abelian $\af_G$-extension of $R$. Thus, we have a  semigroup homomorphism
\[\phi:\prod_{i=i}^n {T_{par}(G_i,R)}\to {T_{par}(G,R)},\quad \phi\big(\lfloor S_1, {\alpha_{G_1}\rfloor,\lfloor S_2, \alpha_{G_2}\rfloor, \dots, \lfloor S_n, \alpha_{G_n}}\rfloor\big)=\lfloor S,{\af_G}\rfloor.\]
To construct the inverse of $\phi$, let $(S,{\alpha_G})$ be a partial action of $G$ on $S$, $R=S^{{\alpha_G}}$ and assume that $R\subset S$ is a partial  abelian $\af_G$-extension. For each $1\leq i\leq n$, we consider  the subgroup
$H_i=G_1\times G_2\times \cdots\times G_{i-1}\times\{1\}\times G_{i+1} \times\cdots \times  G_n$ of $G$. Note that $H_i$ acts partially on $S$ via restriction.  Fix $S_i:=S^{\alpha_{H_i}}$ and observe that $G/H_i\simeq G_i$. By Theorem \ref{pro1.2.6.}, $R\subset S_i$ is a partial  abelian $\af_{G_i}$-extension. Thus, the inverse of $\phi$ is
\[\varphi:{\Hh_{par}(G,R)}\to \prod_{i=i}^n {\Hh_{par}(G_i,R)},\quad \varphi\big(\lfloor S,\af_{G}\rfloor\big)=\big(\lfloor S_1, \alpha_{G_1}\rfloor,\lfloor S_2, \alpha_{G_2}\rfloor, \dots, \lfloor S_n, \alpha_{G_n}\rfloor\big). \]

%
%

\section{On the structure of ${\Hh_{par}(G,R)}$}
It is well-known that a commutative inverse semigroup is a strong semilattice of abelian groups; see for instance Corollary IV.2.2 of \cite{HO}. Let $R$ be an algebra. By  Theorem \ref{invpar},
\begin{align}\label{clifford}
{\Hh_{par}(G,R)=\bigcup\limits_{\xi\in \Lambda} \Hh_{\xi}(G,R)},
\end{align}
with $\Lambda$ a semillatice isomorphic to the set of idempotents of ${\Hh_{par}(G,R)}$ and $\Hh_{\xi}(G,R)$ a group, for all $\xi\in \Lambda$. In order to describe the decomposition of ${\Hh_{par}(G,R)}$ given in \eqref{clifford}, we will investigate the idempotents of ${\Hh_{par}(G,R)}$ in the next subsection.

\subsection{Idempotents of ${\Hh_{par}(G,R)}$} \label{sub:struc}
 We recall that the idempotents of ${\Hh_{par}(G,R)}$ are given by $\lfloor S,\alpha_G \rfloor^{\ast} \ast_{par} \lfloor S,\alpha_G \rfloor$, with $\lfloor S,\alpha_G \rfloor \in {\Hh_{par}(G,R)}$. In order to characterize this elements we introduce some extra notation. For the unital partial action $\alpha_{G}=(S_g,\af_g)_{g\in G}$ of $G$ on $S$ we define:
\begin{align}\label{not-ide1}
&\hat{S}:=\prod_{g\in G}S_g,\qquad\qquad  \hat{S}_{(h,l)}=\prod_{g\in G}S_g1_h1_{l\m g},\quad\text{for all}\,\,\,(h,l)\in G\times G,\\[.2em]
\label{not-ide2}
&\gamma_{(h,l)}\big((a_g1_{h\m}1_{l g})_{g\in G}\big)=\big(\af_h(a_g1_{h\m})1_{hlg}\big)_{g\in G},\quad\text{for all }\,\,(a_g)_{g\in G} \in \hat{S}.
\end{align}
As in the previous section, $\af^\ast_{G}$ denotes the partial action of $G$ on $S$ given by $\af^{\ast}_g=\af_{g\m}$, for all $g\in G$.

\begin{prop}\label{idemp} Let $\alpha_{G}=(S_g,\af_g)_{g\in G}$ be a unital  partial  action of $G$ on $S$ such that 
$R=S^{{\af_{G}}}\subset S$ is a partial $\af_G$-extension. Then the following statements hold:
\begin{enumerate}[\rm (i)]
\item The family of pairs $\gamma_{G\times G}=\big(\hat{S}_{(h,l)},\gamma_{(h,l)}\big)_{(h,l)\in G\times G}$  given by \eqref{not-ide1} and \eqref{not-ide2}
is a unital partial action of $G\times G$ on $\hat{S}$.\smallbreak

\item The  partial action $\theta_{G\times G}:=\af^{\ast}_{G}\otimes \af_{G}$ of $G\times G$ on $S\otimes S$ is partially $(G\times G)$-isomorphic to $\gamma$.
\end{enumerate}
\end{prop}

\begin{proof}
(i) Let $(h,l)\in G\times G$ and $(a_g)_{g\in G} \in \hat{S}$. Notice that
\begin{align*}
\gamma_{(h,l)}\big((a_g1_{h\m}1_{lg})_{g\in G}\big)&=\big((\af_h(a_g1_{h\m})1_{hlg})1_h1_{hg}\big)_{g\in G}\\
&\overset{\mathclap{t=hlg}}{=}\,\,\,\big((\af_h(a_g1_{l\m})1_t)1_h1_{l\m t}\big)_{t\in G},
\end{align*}
which implies that $\gamma_{(h,l)}$ is well-defined. Clearly, $\gamma_{(h,l)}$ is an algebra isomorphism whose inverse is $\gamma_{(h\m,l\m)}$. It is straightforward to check that $\gamma$ is a partial action. \smallbreak

\noindent (ii) Consider $\varphi:S\otimes S\to \hat{S}$ defined by $\varphi(x\ot y)=(x\af_g(y1_{g\m}))_{g\in G}$, for all $x,y\in S$. By  Theorem 4.1 (iv) of \cite{DFP}, $\varphi$ is an algebra isomorphism and it is $R$-linear. Given $(h,l)\in G\times G$, we have
$$\varphi(S_h\otimes S_{l\m})=\prod\limits_{g\in G}S_h\af_g(S_{l\m}1_{g\m})=\prod\limits_{g\in G}S_gS_hS_{gl\m}=\hat{S}_{(h,l)}.$$
Also, if  $x,y\in S$ then
\begin{align*}
(\varphi \circ \theta_{(h,l)})(x1_{h\m}\ot y1_l)&=\varphi\big(af_h(x1_{h\m})\ot \af_{l\m}( y1_l)\big)\\
&=\big(\af_h(x1_{h\m})\af_{gl\m}( y1_{g\m l})1_g\big)_{g\in G}\\
&\overset{\mathclap{(\ast)}}{=}\big(\af_h(x1_{h\m})\af_{ht}( y1_{h\m t\m})1_{hlt}\big)_{t\in G}\\
&=\big(\af_h(x\af_{t}( y 1_{t\m})1_{lt}1_{h\m})\big)_{t\in G},
\end{align*}
where $(\ast)$ follows by taking $t=h\m l\m g$. On the other hand
\begin{align*}
( \gamma_{(h,l)}\circ\varphi )(x1_{h\m}\ot y1_l)&=\hat{\af}_{(h,l)}\big((x1_{h\m} \af_g(y1_{g\m})1_{gl})_{g\in G}\big)\\
&=\gamma_{(h,l)}\big((x\af_g(y1_{g\m})1_{h\m} 1_{gl})_{g\in G}\big)\\
&=\big(\af_{h}(x\af_g(y1_{g\m})1_{h\m} 1_{gl})\big)_{g\in G}.
\end{align*}
Thus $\varphi \circ \theta_{(h,l)}=\gamma_{(h,l)}\circ\varphi$ in $S_{h\m}\ot S_l$, for all $l,t\in G$. Hence $\theta_{G\times G}$ and $\gamma_{G\times G}$ are  partially $(G\times G)$-isomorphic. \end{proof}

Let $\alpha_{G}=(S_g,\af_g)_{g\in G}$ be a partial action of $G$ on $S$ and $(\hat{S},\gamma_{G\times G})$ the partial action given by \eqref{not-ide1} and \eqref{not-ide2}. Denote by $E(S,\af_{G}):=\big(\hat{S}\big)^{\gamma_{\delta G}}$ the subalgebra of invariants  of $\hat{S}$ under $\gamma_{\delta G}$. As in the previous sections, $\tilde{\gamma}_{G}$ is the partial action of $G\simeq (G\times G)/\delta G$ on $E(S,\af_{G})$. In the next proposition we will prove that the classes $\lfloor E(S,\alpha_{G}),\tilde{\gamma}_{G}\rfloor$, where $\lfloor S,\alpha_{G}\rfloor\in {\Hh_{par}(G,R)}$, are the idempotents  of ${\Hh_{par}(G,R)}$ and they will be characterized.

\begin{prop}\label{charac} Let $R$ be an algebra. The idempotents of ${\Hh_{par}(G,R)}$ are the classes $\lfloor E(S,\alpha_{G}),\tilde{\gamma}_{G}\rfloor$, where $\lfloor (S,\alpha_{G})\rfloor\in {\Hh_{par}(G,R)}$. Moreover, if $G=\{1=g_1, g_2, \cdots ,g_m\}$ and $\af_{G}=(S_g,\af_g)_{g\in G}$ then
\begin{align*}
&E(S,\af_{G})=\left\{(a_g)_{g\in G}\in \hat{S}: \af_h(a_g1_{h\m})1_g=a_g1_h1_{hg}, \,\,\, \text{for all }\, h\in G\right\}, \\
&\tilde{\gamma}_{h}((a_g)_{g\in G})=(a_g1_{hg})_{g\in G}+ \sum_{i=2}^m\prod\limits_{j=1}^{i-1}(1_g-1_{hg_{j}}1_{g_jg})_{g\in G}(a_g1_{g_i} 1_{g_ig}1_{hg})_{g\in G},
\end{align*}
for all $h\in G$ and $(a_g)_{g\in G}\in E(S,\alpha_{G})\tilde{1}_{(h\m,1) \delta G}$, where $\tilde{1}_{(h\m,1)\delta G}$ is given by \eqref{1gh}.
\end{prop}

\begin{proof}
By Proposition \ref{idemp}, $(S\ot S,\,\theta_{G\times G})$ and $(\hat{S},\, \gamma_{G\times G})$ are partially $G\times G$-isomorphic. Hence, Theorem \ref{pro10} implies that
$((S\ot S)^{\theta_{\delta G}},\,\theta_{G\times G/\delta G})$  and $(E(S,\alpha_{G}),\, \tilde{\gamma}_{G})$ are partially $G$-iso\-mor\-phic. Thus, $\lfloor E(S,\af_{G}),\tilde{\gamma}_{G}\rfloor=\lfloor (S\ot S)^{\theta_{\delta G}},\,\theta_{G\times G/\delta G}\rfloor=\lfloor S,\alpha_{G} \rfloor^{\ast}\ast_{par}   \lfloor S,\alpha_{G} \rfloor$ are the idempotents of ${\Hh_{par}(G,R)}$. Notice that  $(a_g)_{g\in G}\in E(S,\af_{G})$ if and only if
\[\big(a_g1_{h}1_{h g}\big)_{g\in G}=\gamma_{(h,h\m)}\big((a_g1_{h\m}1_{h\m g})_{g\in G}\big)\stackrel{\eqref{not-ide2}}{=} \big(\alpha_{h}(a_g1_{h\m})1_{g}\big)_{g\in G},\,\,\text{for all}\,\, h\in G.\]

By \eqref{not-ide1}, $\hat{S}_{(h,l)}=\hat{S}\hat{1}_{(h,l)}$, where $\hat{1}_{(h,l)}=(1_g1_h1_{l\m g})_{g\in G}$ for all $h,l\in G$. Consider ${\bf a}=(a_g)_{g\in G}\in E(S,\alpha)\tilde{1}_{(h\m,1) \delta G}$. By \eqref{alphapprima} and \eqref{not-ide2}, we have
\begin{align*}
\tilde{\gamma}_{h}({\bf a})&=\gamma_{(1,h)}({\bf a}\hat{1}_{(1,h\m)})+\sum_{i=2}^m\prod\limits_{j=1}^{i-1}((1_g)_{g\in G}-\hat{1}_{(g_{j}, hg\m_{j})})\gamma_{(g_i, hg\m_i)}({\bf a}\hat{1}_{(g_i\m, h\m g_i)})\\
&=(a_g1_{hg})_{g\in G}+ \sum_{i=2}^m\prod\limits_{j=1}^{i-1}((1_g)_{g\in G}-(1_{g_{j}}1_{h\m g_jg})_{g\in G})(\af_{g_i}(a_g1_{g_i\m} )1_{hg})_{g\in G})\\
&=(a_g1_{hg})_{g\in G}+ \sum_{i=2}^m\prod\limits_{j=1}^{i-1}((1_g)_{g\in G}-(1_{hg_{j}}1_{g_jg})_{g\in G})(a_g1_{g_i} 1_{g_ig}1_{h g})_{g\in G}\\
&=(a_g1_{hg})_{g\in G}+ \sum_{i=2}^m\prod\limits_{j=1}^{i-1}(1_g-1_{hg_{j}}1_{g_jg})_{g\in G}(a_g1_{g_i} 1_{g_ig}1_{h g})_{g\in G}.
\end{align*}
\end{proof}

In order to simplify the notation, denote by $I(S,\af_{G})$ the idempotent  $\lfloor E(S,\alpha_{G}),\tilde{\gamma}\rfloor$ of ${T_{par}(G,R)}$ associated to $\lfloor S,\alpha_{G}\rfloor$. By Proposition \ref{charac} and \eqref{clifford} we have that
\[{\Hh_{par}(G,R)}=\bigcup\limits_{I(S,\alpha_{G})}T_{I(S,\alpha_{G})}(G,R).\]
Observe that $\Hh_{I(S,\alpha_{G})}(G,R)$ is a group with identity element $I(S,\alpha_{G})$.  Moreover
\begin{equation*}
\lfloor S', \alpha'_{I(S,\alpha_{G})}\rfloor \in {T_{I(S,\alpha)}(G,R)}\,\,\Longleftrightarrow \,\,I(S',\alpha'_{G})=I(S,\alpha_{G}).
\end{equation*}

\begin{prop} Let $\lfloor S, \alpha_{G}\rfloor,\,\, \lfloor S', \alpha'_{G}\rfloor \in {\Hh_{par}(G,R)}$.  Then
$\lfloor S', \alpha'_{G}\rfloor\in {\Hh_{I(S,\alpha_{G})}(G,R)}$ if and only if  there exist an algebra homomorphism $f: I(S,\alpha_{G})\to I(S', \alpha'_{G})$ which is $R$-linear and such that $f( (1_g)_{g\in G})=(1'_g)_{g\in G}$.
\end{prop}
\begin{proof} Suppose that $f: I(S,\alpha_{G})\to I(S', \alpha'_{G})$ is an algebra homomorphism $R$-linear such that $f( (1_g)_{g\in G})=(1'_g)_{g\in G}$. Then, \eqref{1gh} implies that $f(\tilde{1}_{(1,g)\delta G})=\tilde{1'}_{(1,g)\delta G}$ and whence  $f\big(I(S,\alpha_{G})\tilde{1}_{(1,g)\delta G}\big)\subseteq I(S',\alpha'_{G})\tilde{1'}_{(1,g)\delta G}$, for all $g\in G$. It is immediate from Proposition \ref{charac} that $f$ satisfies the other condition of \eqref{def:par-G-iso} and consequently the result follows from Proposition \ref{pro9}. The converse is immediate.
\end{proof}

In the next result we present sufficient conditions for $E(S,\af_{G})$ to be a (global)  $\af_G$-extension of $R$.

\begin{prop}\label{pro-fim}
	Let $(S,\af_{G})$ be a unital partial action of $G$ on $S$. If $\anula_R(1_g)=0$ for all $g\in G$, then $E(S,\af_{G})$ is a (global) $\af_G$-extension of $R$. In this case, we have that $\lfloor S,\af_{G}\rfloor\in {T_{I(E_G(R),\rho_G)}(G,R)}$.
\end{prop}
\begin{proof}
For each $g\in G$, consider the surjective algebra homomorphism $\pi_g:\hat{S}\to S_g$, given by $\pi_g\big((a_g)\big)=a_g$. Denote by $E_g(S,\alpha_{G}):=\pi_g\big(E(S,\alpha_{G})\big)$. Since $E(S,\af_{G})$ is a partial $\af_G$-extension of $R$, it follows from Theorem 4.1 of \cite{DFP} that $E(S,\af_{G})$ is a finitely generated projective $R$-module. It is clear that $E_g(S,\alpha_{G})$ is a finitely generated $R$-module. Also, from $E(S,\alpha_{G})=\prod_{g\in G}E_g(S,\alpha_{G})$  it follows that $E_g(S,\alpha_{G})$ is a projective $R$-module. Note that the map $\iota_g: R\to E_g(S,\alpha_{G})$, defined by $\iota_g(r)=r1_g$ for all $r\in R$, is a monomorphism of algebras. Thus, $\operatorname{rank}_{R_{\wp}} E_g(S,\alpha_{G})_{\wp}\geq 1$, for any prime ideal $\wp$ of $R$. Hence
\[|G|\leq \sum_{g\in G} \operatorname{rank}_{R_{\wp}} E_g(S,\alpha_{G})_{\wp}=\operatorname{rank}_{R_{\wp}} E(S,\alpha_{G})_{\wp}.\]
By Corollary 4.6 of \cite{DFP}, $\operatorname{rank}_{R_{\wp}} E(S,\alpha_{G})_{\wp}\leq |G|$ and whence $\operatorname{rank}_{R_{\wp}} E(S,\alpha_{G})_{\wp}=|G|$, for any prime ideal $\wp$ of $R$. Using Corollary 4.6 of \cite{DFP} again, we conclude that $E(S,\alpha_{G})$ is a global  $\af_G$-extension of $R$. The last  assertion of the proposition is obvious.
\end{proof}

\begin{rem}{\rm  Let $[E_G(R),\rho_{G}]$ be the identity element of $\Hh(G,R)$ given in  subsection \ref{harglo}. It is clear that $\Hh(G,R)$ is a subgroup of ${\Hh_{I(E_G(R),\rho_G)}(G,R)}$. 
In general, it is not true that $\Hh(G,R)={\Hh_{I(E_G(R),\rho_G)}(G,R)}$. In fact, let $(S,\alpha_G)$ be the partial action of $G=C_4$ on $S=Re_1\oplus Re_2\oplus Re_3$ given in Example \ref{ex0}. It is immediate to check that $S$ is a partial $\alpha_G$-extension of $R$ and $\anula_R(1_g)=0$, for all $g\in G$. By Proposition \ref{pro-fim}, $\lfloor S,\af_{G}\rfloor\in {\Hh_{I(E_G(R),\rho_G)}(G,R)}$. However, $\lfloor S,\af_{G}\rfloor\notin \Hh(G,R)$ because $(S,\alpha_G)$ is not global.}
\end{rem}


\begin{thebibliography}{99}

\bibitem{AG} M. Auslander and O. Goldman, \emph{The Brauer group of a commutative ring},
Trans. Amer. Math. Soc. 97 (1960), 367-408.

\bibitem{CHR} S. U. Chase, D. K. Harrison, A. Rosenberg, \emph{Galois theory and Galois cohomology of comutative rings}, Mem. Amer. Math. Soc. 52 (1968), 1-19.

\bibitem{DE} M. Dokuchaev and R. Exel, \emph{Associativity of crossed products by partial actions, enveloping actions and partial representations},
Trans. Amer. Math. Soc. 357 (2005), 1931-1952.
	
\bibitem{DFP} M. Dokuchaev, M. Ferrero and A. Paques, \emph{Partial actions and Galois theory}, J. Pure Appl. Algebra 208 (2007), 77-87.
	
\bibitem{DPP} M. Dokuchaev,  A. Paques and H. Pinedo, \emph{Partial Galois cohomology and related homomorphism}, Quart. J. Math. 70 (2019), 737-766.
	
\bibitem{H} D. K. Harrison, \emph{Abelian extensions of commutative rings}, Mem. Amer. Math. Soc. 52 (1965), 1-14.
	
\bibitem{HO} J. M. Howie, \emph{An introduction to semigroup theory}, Academic Press, London, NY, 1976.
	
\bibitem{L} M. V. Lawson, \emph {Inverse semigroups: the theory of partial symmetries}, World Scientific, 1998.

\bibitem{PRS}A. Paques, V. Rodrigues and A. Sant'Ana, \emph{Galois correspondences for partial Galois Azumaya
extensions}, J. Algebra Appl., 10 (5) (2011), 835-847.
	
\end{thebibliography}
\end{document}